\documentclass[reqno]{amsart}
\usepackage{amssymb, amsfonts}
\usepackage{enumerate}
\usepackage[usenames, dvipsnames]{color}
\usepackage{hyperref}

\usepackage{verbatim}

\numberwithin{equation}{section}

\newtheorem{theorem}{Theorem}[section]
\newtheorem{corollary}[theorem]{Corollary}
\newtheorem{lemma}[theorem]{Lemma}
\newtheorem{proposition}[theorem]{Proposition}

\newtheorem*{AssumpA'}{Assumption A$'$ $(\rho,\varepsilon)$}
\newtheorem*{AssumpA''}{Assumption A$''$ $(\rho,\varepsilon)$}

\theoremstyle{definition}
\newtheorem{remark}[theorem]{Remark}

\theoremstyle{definition}

\theoremstyle{definition}
\newtheorem{assumption}[theorem]{Assumption}

\makeatletter
\def\dashint{\operatorname%
{\,\,\text{\bf--}\kern-.98em\DOTSI\intop\ilimits@\!\!}}
\makeatother

\def\bC{\mathbb{C}}

\def\bR{\mathbb{R}}
\def\bZ{\mathbb{Z}}

\def\ff{\mathfrak{f}}
\def\fg{\mathfrak{g}}
\def\fh{\mathfrak{h}}

\def\cA{\mathcal{A}}
\def\cB{\mathcal{B}}

\def\cD{\mathcal{D}}

\def\cL{\mathcal{L}}
\def\cM{\mathcal{M}}

\def\cU{\mathcal{U}}

\begin{document}

\title[Stokes system]{Weighted $L_q$-estimates for stationary Stokes system with partially BMO coefficients}

\author[H. Dong]{Hongjie Dong}
\address[H. Dong]{Division of Applied Mathematics, Brown University, 182 George Street, Providence, RI 02912, USA}

\email{Hongjie\_Dong@brown.edu}

\thanks{H. Dong was partially supported by the NSF under agreement DMS-1056737 and DMS-1600593.}

\author[D. Kim]{Doyoon Kim}
\address[D. Kim]{Department of Mathematics, Korea University, 145 Anam-ro, Seongbuk-gu, Seoul, 02841, Republic of Korea}

\email{doyoon\_kim@korea.ac.kr}

\thanks{D. Kim was supported by Basic Science Research Program through the National Research Foundation of Korea (NRF) funded by the Ministry of Education (2016R1D1A1B03934369).}



\begin{abstract}
We prove the unique solvability of solutions in Sobolev spaces to the stationary Stokes system on a bounded Reifenberg flat domain when the coefficients are partially BMO functions, i.e., locally they are merely measurable in one direction and have small mean oscillations in the other directions.
Using this result, we establish the unique solvability in Muckenhoupt type weighted Sobolev spaces for the system with partially BMO coefficients on a Reifenberg flat domain.
We also present weighted a priori $L_q$-estimates for the system when the domain is the whole Euclidean space or a half space.
\end{abstract}

\maketitle

\section{Introduction}

In this paper, we continue our study \cite{DK15S} on stationary Stokes system with rough coefficients, where we considered a system in Sobolev spaces whose coefficients are merely measurable in one direction, i.e., they can be very irregular (no regularity assumptions) in one direction.
The stationary Stokes system we consider contains a second-order divergence type operator with variable coefficients:
\begin{equation}
                                        \label{eq11.08}
\begin{cases}
\cL u + \nabla p = f + D_\alpha f_\alpha
\quad
&\text{in}\,\,\Omega,
\\
\operatorname{div} u = g
\quad
&\text{in}\,\,\Omega,
\end{cases}
\end{equation}
where $\Omega  \subseteq \bR^d$ and $\cL$ is defined by
\begin{equation}
							\label{eq0812_01}
\cL u = D_\alpha \left( A^{\alpha\beta} D_\beta u \right),
\quad A^{\alpha\beta} = [A^{\alpha\beta}_{ij}]_{i,j=1}^d
\end{equation}
for $\alpha,\beta =1,\ldots,d$. Throughout this paper, we use the Einstein summation convention on repeated indices.
The coefficients $A^{\alpha\beta}$, as functions of $x \in \bR^d$, are bounded and satisfy the strong ellipticity condition. See \eqref{eq1016_01}.

As mentioned in \cite{DK15S}, besides its mathematical interests, such a system is also partly motivated by the study of inhomogeneous fluids with density dependent viscosity (see, for instance, \cite{LS75, Li96, AGZ11}), as well as equations describing flows of shear thinning and shear thickening fluids with viscosity depending on pressure (see, for instance, \cite{FMR05, BMR07}). It also has a connection to the Navier-Stokes system in general Riemannian manifolds. See, for instance, \cite{DM04}. Since the coefficients are merely measurable in one direction, they may have jump discontinuities and hence, the system can be used to model, for example, the motion of two fluids with interfacial boundaries.

In \cite{DK15S}, we established a priori $L_q$-estimates, $q \in [2,\infty)$, for \eqref{eq11.08} when the domain $\Omega$ is the whole Euclidean space $\bR^d$ or a half space $
\bR^d_+$ under the assumption that $A^{\alpha\beta}$ are functions of only one variable with no regularity assumptions.
We also proved there an $L_q$-estimate, $q \in (1,\infty)$, and the unique solvability of \eqref{eq11.08} in Sobolev spaces when $\Omega$ is a bounded Lipschitz domain with a small Lipschitz constant.
In this case, the coefficients $A^{\alpha\beta}$ are assumed to be merely measurable in one direction, which is almost parallel to the normal direction near the boundary, and have small bounded mean oscillations (BMO) in the other directions.
This type of coefficients is called (variably) partially BMO coefficients and was first introduced in \cite{MR2540989}.
See Assumption \ref{assum0711_1}.
For other previous results on the Stokes system and discussions about the system with variable coefficients, we refer the reader to \cite{MR1313554, MR975121, MR641818, DK15S} and the references therein.

In this paper, we generalize the results of \cite{DK15S} in two respects.
First, we consider the stationary Stokes system defined on a more general domain, called a Reifenberg flat domain. See Theorem \ref{thm1}.
The solution spaces are standard Sobolev spaces (without weights) as in \cite{DK15S}.
In this study, we not only deal with more general domains, but also clearly identify the classes of coefficients as well as domains for the existence and uniqueness of solutions to the stationary Stokes system in Sobolev spaces.
As is shown, for instance, in \cite{MR1331981}, even the Poisson equation on a Lipschitz domain may not be solvable in $W_q^1(\Omega)$ with $q > 3$ ($q > 4$ if $d=2$) unless the boundary is sufficiently flat.
Likewise, if coefficients $A^{\alpha\beta}$ are in the class of partially BMO functions, the sizes of the mean oscillations of $A^{\alpha\beta}$ on small balls need to be sufficiently small.
Indeed, suppose that both quantities (the flatness and the mean oscillations) are bounded by a positive number $\rho$.
Then, in elliptic and parabolic cases with partially BMO coefficients on Reifenberg flat domains (see, for instance, \cite{DK15}),
there exists a unique solution to a given equation in a Sobolev space if $\rho$ is sufficiently small.
It is important that the size of $\rho$ is determined only by parameters such as the dimension, the ellipticity constant, and $q$ if solutions are to be found in a $L_q$-based Sobolev space.
In the case of the Stokes system, it is more involved to determine $\rho$ because of the divergence equation $\operatorname{div} u = g$.
Here, we clarify the dependence of $\rho$, which amounts to identifying possible coefficients and domains for our main results.
See Assumptions \ref{assum1004}, \ref{assum0711_1}, and Remark \ref{rem0116_1}.
It is worth noting that in a recent work \cite{CL15}, an a priori estimate of solutions was proved for the stationary Stokes system on a Reifenberg flat domain, and the solvability was mentioned for Lipschitz domains, when the coefficients have small mean oscillations with respect to {\em all} the variables.

Second, we extend the results in \cite{DK15S} to the framework of Sobolev spaces with weights (see Theorem \ref{thm3}), under the same regularity assumptions on the coefficients and the boundary of the domain as in the unweighted case (see Theorem \ref{thm1}).
We consider Muckenhoupt weights, $\omega \in A_q$, $q \in (1,\infty)$, which are defined on the same domain as the Stokes system.
See \eqref{eq0116_02}.
In a recent work \cite{BS16}, the authors studied the stationary Stokes system in Sobolev spaces having BMO coefficients as in \cite{CL15}.
They considered Muckenhoupt weights in $A_{q/2}$ with $q\in (2,\infty)$ defined in the whole Euclidean space.
Owing to the fact that $A_{q/2} \subsetneq A_q$ and considering the Hardy-Littlewood maximal function theorem with $A_q$ weights, our results are sharp in the sense that we allow $A_q$ weights in obtaining $L_q$-estimates for any $q\in (1,\infty)$.
Moreover, we consider weights defined on domains, which might be more general as the restriction of a Muckenhoupt weight in $\bR^d$ to a domain with an exterior measure condition, like a Reifenberg flat domain, results in a Muckenhoupt weight on the domain.
We remark that we also deal with the stationary Stokes system defined in the whole Euclidean space and on a half space using Muckenhoupt weights, while, as mentioned above, in \cite{DK15S} we only treated this case for $q \in [2,\infty)$ without weights.
See Theorem \ref{thm3}.

The regularity assumptions in this paper on coefficients and domains have also been considered in recent papers, \cite{MR2835999, DK15} for instance, on elliptic and parabolic equations/systems.
In particular, as far as coefficients are concerned, the assumption allowing coefficients to be merely measurable in one direction cannot be relaxed in view of the counterexamples about the unique solvability of elliptic equations in Sobolev spaces (see \cite{MR0159110, MR3266252}) when the coefficients are only measurable functions of two variables.
See \cite{DK15} and the references therein for a comprehensive study on elliptic and parabolic equations/systems in Sobolev space with Muckenhoupt weights.

To prove our main results, we take two different approaches to the unweighted and weighted cases.
When dealing with the stationary Stokes system in Sobolev spaces without weights, we use a level-set type of argument as used in \cite{MR2835999} to obtain a desired $L_q$-estimate.
In short, by measuring the level sets of a solution in $L_2$, we show that the solution is indeed in $L_q$, $q>2$, if the terms on the right-hand side of the system are in $L_q$.
Since the solution is not known to be in $L_q$ a priori, the main ingredient of the proof is a reverse H\"{o}lder's inequality, for which we utilize the $L_2$-estimate of the Stokes system accompanied by that of the divergence equation.
Thus, when determining the size of $\rho$ later (see Assumption \ref{assum0711_1}), we need to take into consideration the constant in the $L_2$-estimate for the divergence equation, which may carry some information about the domain of the system.
See the proof of Lemma \ref{lem0712_1} and Remark \ref{rem0118_1}.

For the proofs of the main results for Sobolev spaces with weights, we make use of the mean oscillation estimate approach presented, for instance, in \cite{MR2771670, MR2835999, DK15}.
As explained in \cite{MR2771670}, mean oscillation estimates are well suited to the perturbation argument for coefficients having small mean oscillations.
Moreover, as shown in \cite{DK15}, they are in an appropriate form to deal with $L_q$-estimates with Muckenhoupt weights via the Hardy-Littlewood maximal function theorem and the Fefferman-Stein theorem on sharp functions, which are valid for $L_q$ spaces with Muckenhoupt weights.
The mean oscillation estimates rely on Theorem \ref{thm1}, which is for the unweighted case. Hence, this cannot be considered as a simple consequence of our main results for the weighted case, in particular, Theorem \ref{thm2} with the weight $\omega \equiv 1$.

The remainder of this paper is organized as follows. We state the main results of this paper along with some notation description and assumptions in the following section.
In Section \ref{sec03}, we prove reverse H\"{o}lder's inequality for solutions in $L_2$ to the Stokes system with general coefficients.
In Section \ref{sec04}, we prove interior and boundary $L_\infty$ and H\"{o}lder estimates for derivatives of solutions when the system has coefficients measurable in one direction.
Using the results obtained in the previous sections, we prove the existence and uniqueness of solutions together with $L_q$-estimates to the Stokes system defined on a Reifenberg flat domain with partially BMO coefficients in Section \ref{sec05}.
The solution spaces here are Sobolev spaces without weights.
Finally, we devote Section \ref{sec06} to the proofs of Theorems \ref{thm3} and \ref{thm2}, which are for Sobolev spaces with Muckenhoupt weights.

\section{Main results}

Before we present our main results, we introduce some notation used throughout the paper and state the assumptions for our main theorems.
We fix a half space to be $\bR^d_+$, defined by
$$
\bR^d_+ = \{ x = (x_1, x') \in \bR^d: x_1 > 0,\, x' \in \bR^{d-1} \}.
$$
Let $B_r(x_0)$ be an Euclidean ball of radius $r$ in $\bR^d$ centered at $x_0\in \bR^d$, and let $B^+_r(x_0)$ be the half ball
$$
B_r^+(x_0) = B_r(x_0) \cap \bR^d_+.
$$
A ball in $\bR^{d-1}$ is denoted by
$$
B'_r(x') = \{ y' \in \bR^{d-1}: |x' - y'| < r\}.
$$
We use the abbreviations
$B_r := B_r(0)$, $B_r^+ := B_r^+(0)$ where $0\in \bR^d$, and $B_r' := B_r'(0)$ where $0 \in \bR^{d-1}$.
We also denote
$$
(f)_\cD = \frac{1}{|\cD|} \int_{\cD} f \, dx = \dashint_{\cD} f \, dx,
$$
$$
\Omega_r(x_0) = \Omega \cap B_r(x_0).
$$

Throughout the paper, the coefficients $A^{\alpha\beta}$ are assumed to be bounded and satisfy the strong ellipticity condition, i.e., there exists a constant $\delta \in (0,1)$ such that
\begin{equation}
							\label{eq1016_01}
|A^{\alpha\beta}| \le \delta^{-1},
\quad
\sum_{\alpha, \beta=1}^d \xi_\alpha \cdot A^{\alpha\beta} \xi_\beta \ge \delta \sum_{\alpha=1}^d |\xi_\alpha|^2
\end{equation}
for any $\xi_\alpha \in \bR^d$, $\alpha = 1, \ldots, d$.

We say that $(u,p)\in W_q^1(\Omega)^d\times L_q(\Omega)$ is a solution to \eqref{eq11.08} if we have 
$$
- \int_\Omega D_\alpha \psi \cdot A^{\alpha\beta} D_\beta u \, dx - \int_\Omega p \, \operatorname{div} \psi \, dx = \int_\Omega f \cdot \psi \, dx - \int_\Omega f_\alpha \cdot D_\alpha \psi \, dx
$$
for any $\psi = (\psi_1, \ldots, \psi_d) \in C_0^\infty(\Omega)^{d}$,
where
$$
\operatorname{div} \psi = D_1 \psi_1 + \ldots + D_d \psi_d
\quad
\text{and}
\quad
D_\alpha \psi = (D_\alpha \psi_1, D_\alpha \psi_2, \ldots, D_\alpha \psi_d).
$$

We assume that $\Omega$ is a Reifenberg flat domain in the following sense.

\begin{assumption}
							\label{assum1004}
There exists $R_0 \in (0,\infty)$ such that, for any $x_0 \in \partial \Omega$ and $0 < r \le R_0$,
there is a coordinate system depending on $x_0$ and $r$ such that in the new coordinate system we have
\begin{equation*}
\left\{ (y_1, y') : {x_0}_1 + \frac{r}{16} < y_1 \right\} \cap B_r(x_0) \subset \Omega_r(x_0)
\subset \left\{(y_1,y') : {x_0}_1 - \frac{r}{16} < y_1\right\} \cap B_r(x_0),
\end{equation*}
where ${x_0}_1$ is the first coordinate of $x_0$ in the new coordinate system.
\end{assumption}

Next, we state our assumption on the regularity of the coefficients $A^{\alpha\beta}$ combined with the flatness of $\partial\Omega$.
Let $\rho \in (0,1/16)$ be a small constant to be specified later.

\begin{assumption}[$\rho$]
							\label{assum0711_1}
There exists $R_1 \in (0,R_0]$ satisfying the following.

\noindent
(i) For $x_0 \in \Omega$ and $0 < r \le \min\{R_1, \operatorname{dist}(x_0, \partial\Omega)\}$, there is a coordinate system depending on $x_0$ and $r$ such that in this new coordinate system we have
\begin{equation}
							\label{eq0307_01}
\dashint_{B_r(x_0)} \Big| A^{\alpha\beta}(x_1,x') - \dashint_{B_r'(x_0')} A^{\alpha\beta}(x_1,z') \, dz' \Big| \, dx \le \rho.
\end{equation}

\noindent
(ii) For any $x_0 \in \partial \Omega$ and $0 < r \le R_1$,
there is a coordinate system depending on $x_0$ and $r$ such that in the new coordinate system \eqref{eq0307_01} holds, and
\begin{equation*}
\{ (y_1, y') : {x_0}_1 + \rho r < y_1 \} \cap B_r(x_0) \subset \Omega_r(x_0)
\subset \{(y_1,y') : {x_0}_1 - \rho r < y_1\} \cap B_r(x_0),
\end{equation*}
where ${x_0}_1$ is the first coordinate of $x_0$ in the new coordinate system.
\end{assumption}

\begin{remark}
							\label{rem0116_1}
Clearly, Assumption \ref{assum0711_1} is stronger than Assumption \ref{assum1004}.
The reason we state these two assumptions separately is the following.
Owing to the role of the divergence equation in Stokes systems, in the main theorems below, $\rho$ is determined by a set of parameters including the boundary flatness and $R_0$ in Assumption \ref{assum1004}.
Without Assumption \ref{assum1004}, the size of $\rho$ needs to be chosen only using the information in Assumption \ref{assum0711_1}, which may lead to a circular reasoning.
For instance, even if $A^{\alpha\beta}$ are uniformly continuous, they may not satisfy \eqref{eq0307_01} for a pair of $R_1$ and $\rho$ if $\rho$ is given by $R_1$.
With two assumptions as above, uniformly continuous coefficients satisfy the condition \eqref{eq0307_01} for any given $\rho$ by choosing a sufficiently small $R_1$.
\end{remark}

Our first result is the $L_q$-estimate for the Stokes system.
\begin{theorem}
							\label{thm1}
Let $q ,q_1\in (1,\infty)$ satisfy $q_1 \ge  q d/(q+d)$, $K > 0$, and let $\Omega$ be bounded with $\operatorname{diam}\Omega \le K$.
Then, there exists a constant $\rho = \rho(d,\delta, R_0, K,q) \in (0,1/16)$ such that, under Assumptions \ref{assum1004} and \ref{assum0711_1} $(\rho)$, for $(u,p) \in  W_q^1(\Omega)^d \times L_q(\Omega)$
satisfying $(p)_\Omega = 0$ and
\begin{equation}
							\label{eq0307_04}
\begin{cases}
\cL u + \nabla p = f + D_\alpha f_\alpha
\quad
&\text{in}\,\,\Omega,
\\
\operatorname{div} u = g
\quad
&\text{in}\,\,\Omega,
\\
u = 0
\quad
&\text{on} \,\, \partial \Omega,
\end{cases}
\end{equation}
where $f \in L_{q_1}(\Omega)$, $f_\alpha, g \in L_q(\Omega)$, we have \begin{equation}
							\label{eq0328_02}
\|Du\|_{L_q(\Omega)} + \|p\|_{L_q(\Omega)} \le N \left( \|f\|_{L_{q_1}(\Omega)} + \|f_\alpha\|_{L_q(\Omega)} + \|g\|_{L_q(\Omega)} \right),
\end{equation}
where $N>0$ is a constant depending only on $d$, $\delta$, $R_0$, $R_1$, $K$, $q$, and $q_1$.
Moreover, for $f \in L_{q_1}(\Omega)$, $f_\alpha, g \in L_q(\Omega)$ with $(g)_\Omega = 0$, there exists a unique $(u,p) \in W_q^1(\Omega)^d \times L_q(\Omega)$ satisfying $(p)_\Omega = 0$ and \eqref{eq0307_04}.
\end{theorem}

\begin{remark}
							\label{rem0918_1}
Theorem \ref{thm1} states that the choice of $\rho$ depends not only on $d$, $\delta$, and $q$, but also on $R_0$ and $K$.
Indeed, in the proof of Theorem \ref{thm1} the constant $\rho$ is determined by a set of parameters including $K_1$. The latter, in turn, is given by $d$, $R_0$, and $K$ through the $L_2$-estimate for divergence equations on John domains.
See Assumption \ref{assum0224_1} and Remark \ref{rem0229_2} below.
\end{remark}

We also consider the Stokes system in weighted spaces with Muckenhoupt weights. To this end, we introduce some additional notation. For any $q\in (1,\infty)$, let $A_q=A_q(\Omega)$ be the set of all nonnegative $L_{1,\text{loc}}$ functions $\omega$ on $\Omega$ such that
\begin{equation}
							\label{eq0116_02}
[\omega]_{A_q}:=\sup_{x_0\in \Omega,r>0}\left(\dashint_{\Omega_r(x_0)}\omega(x)\,dx\right)
\left(\dashint_{\Omega_r(x_0)}\big(\omega(x)\big)^{-1/(q-1)}\,dx\right)^{q-1}<\infty.
\end{equation}
We write $f \in L_{q,\omega}(\Omega)$ if
$$
\int_\Omega |f|^q \omega \, dx < \infty.
$$
We also use $\omega(\cdot)$ to denote the measure $\omega(dx)=\omega\,dx$, i.e., for $A \subset \Omega$,
\begin{equation*}
\omega(A) = \int_A \omega(x) \, dx.
\end{equation*}
We use the following weighted Sobolev spaces,
$$
W_{q,\omega}^1(\Omega) = W_q^1(\Omega, \omega \, d x) = \{ u : u, Du \in L_{q,\omega}(\Omega) \}.
$$

The next theorem is a generalization of \cite[Theorem 2.1]{DK15S}, in which a unweighted $L_q$-estimate with $q\ge 2$ was obtained.

\begin{theorem}
							\label{thm3}
Let $q \in (1,\infty)$, $\Omega$ be either $\bR^d$ or $\bR^d_+$, $A^{\alpha\beta} = A^{\alpha\beta}(x_1)$, and $\omega\in A_q(\Omega)$.
If $(u,p) \in W_{q,\omega}^1(\Omega)^d \times L_{q,\omega}(\Omega)$ satisfies
\begin{equation*}
\begin{cases}
\cL u + \nabla p = D_\alpha f_\alpha
\quad
&\text{in}\,\,\Omega,
\\
\operatorname{div} u = g
\quad
&\text{in}\,\,\Omega,
\\
u = 0
\quad
&\text{on} \,\, \partial \Omega \quad \text{in case} \,\, \Omega = \bR^d_+,
\end{cases}
\end{equation*}
where $f_\alpha, g \in L_{q,\omega}(\Omega)$,
then we have
\begin{equation}
							\label{eq0307_03}
\|Du\|_{L_{q,\omega}(\Omega)} + \|p\|_{L_{q,\omega}(\Omega)} \le N \left( \|f_\alpha\|_{L_{q,\omega}(\Omega)} + \|g\|_{L_{q,\omega}(\Omega)} \right),
\end{equation}
where $N=N(d,\delta,q,[\omega]_{A_q})$.
\end{theorem}

\begin{theorem}
							\label{thm2}
Let $q\in (1,\infty)$ and $K > 0$ be constants, $\Omega$ be bounded with $\operatorname{diam}\Omega \le K$, and $\omega\in A_q(\Omega)$.
Then, there exists a constant 
$$
\rho = \rho(d,\delta,R_0, K,q,[\omega]_{A_q}) \in (0,1/16)
$$
such that, under Assumptions \ref{assum1004} and \ref{assum0711_1} $(\rho)$, for $(u,p) \in  W_{q,\omega}^1(\Omega)^d \times L_{q,\omega}(\Omega)$
satisfying $(p)_\Omega = 0$ and
\begin{equation}
							\label{eq11.28}
\begin{cases}
\cL u + \nabla p = D_\alpha f_\alpha
\quad
&\text{in}\,\,\Omega,
\\
\operatorname{div} u = g
\quad
&\text{in}\,\,\Omega,
\\
u = 0
\quad
&\text{on} \,\, \partial \Omega,
\end{cases}
\end{equation}
where $f_\alpha, g \in L_{q,\omega}(\Omega)$, we have
\begin{equation}
							\label{eq11.32}
\|Du\|_{L_{q,\omega}(\Omega)} + \|p\|_{L_{q,\omega}(\Omega)} \le N \left( \|f_\alpha\|_{L_{q,\omega}(\Omega)} + \|g\|_{L_{q,\omega}(\Omega)} \right),
\end{equation}
where $N>0$ is a constant depending only on $d$, $\delta$, $R_0$, $R_1$, $K$, $q$, $[\omega]_{A_q}$, and $\omega(\Omega)$.
Moreover, for $f_\alpha, g \in L_{q,\omega}(\Omega)$ with $(g)_\Omega = 0$, there exists a unique $(u,p) \in W_{q,\omega}^1(\Omega)^d \times L_{q,\omega}(\Omega)$ satisfying $(p)_\Omega = 0$ and \eqref{eq11.28}.
\end{theorem}

\begin{remark}
A mixed norm version of the results in Theorem \ref{thm2} can be easily derived from Theorem \ref{thm2} and a version of the Rubio de Francia extrapolation theorem.
See \cite{DK15}.
\end{remark}

\section{Reverse H\"{o}der's inequality}
							\label{sec03}

Note that in this section we impose no regularity assumptions on the coefficients $A^{\alpha\beta}$ of the elliptic operator $\cL$ in \eqref{eq0812_01}.

\begin{assumption}
							\label{assum0224_1}
There exists a constant $K_1 > 0$ such that, for any $g \in L_{2}(\Omega)$ with $\int_{\Omega} g \, dx = 0$, there exists $B g \in \mathring{W}_2^1(\Omega)^d$ satisfying
\begin{equation}
							\label{eq0711_00}
\operatorname{div} Bg = g
\quad
\text{in}
\,\,
\Omega,
\quad
\|D(Bg)\|_{L_2(\Omega)} \le K_1 \|g\|_{L_2(\Omega)}.	
\end{equation}
\end{assumption}

\begin{remark}
							\label{rem0229_1}
If $\Omega = B_R$ or $\Omega = B_R^+$, it follows from a scaling argument that the constant $K_1$ depends only on the dimension $d$.
The same result holds when $L_2$ is replaced by $L_q$ with $q \in (1,\infty)$, in which case the constant also depends on $q$.
\end{remark}

\begin{remark}
							\label{rem0229_2}
If $\Omega$ is a bounded Reifenberg flat domain that satisfies Assumption \ref{assum1004}, then the domain is also a John domain.
From the result in \cite{MR2263708}, we know that the domain $\Omega$ satisfies Assumption \ref{assum0224_1} with a constant $K_1$ depending only on $d$, $R_0$, and $\operatorname{diam}(\Omega)$.
Indeed, for a bounded domain $\Omega$, we have
$$
\text{Reifenberg flat domains} \subset \text{NTA domains} \subset \text{uniform domains} \subset \text{John domains},
$$
where the first inclusion is proved in \cite[Theorem 3.1]{MR1446617}, and the second and third inclusions follow from the definitions.
For the definition of NTA domains, see \cite[Definition 2.2]{MR1446617} and \cite[Definition 2.3]{MR2135732}.
For (different versions of) definitions of uniform domains (or 1-sided NTA domains), see \cite[Definition 2.2]{MR2135732} and \cite[Definitions 2.12 and 2.14]{AHMNT2014}.
In particular, the equivalence of the two different definitions of uniform domains is explained in the paragraph following \cite[Definition 2.14]{AHMNT2014}.
The inclusion of uniform domains into John domains follows readily if one uses the definitions in \cite[Definitions 2.1 and 2.2]{MR2135732}, where Definition 2.1 is that of John domains.
\end{remark}

Note that, in the following two lemmas, we only impose Assumption \ref{assum0224_1} on the domain.

\begin{lemma}
							\label{lem0225_1}
Let $q_1 \in (1,\infty)$ with $q_1 \ge 2d/(d+2)$, $\Omega \subset \bR^d$ be a bounded domain with $\operatorname{diam} \Omega \le K$ satisfying Assumption \ref{assum0224_1}, and $f \in L_{q_1}(\Omega)$, $f_\alpha, g \in L_2(\Omega)$ with $(g)_\Omega = 0$.
Then, there exists a unique $(u,p) \in W_2^{1}(\Omega)^d \times L_2(\Omega)$ with $(p)_\Omega = 0$ satisfying
$$
\begin{cases}
\cL u + \nabla p = f + D_\alpha f_\alpha
\quad
&\text{in}\,\,\Omega,
\\
\operatorname{div} u = g
\quad
&\text{in}\,\,\Omega,
\\
u = 0\quad &\text{on}\,\,\partial \Omega.
\end{cases}
$$
Moreover, we have
$$
\|Du\|_{L_2(\Omega)} + \|p\|_{L_2(\Omega)} \le N_1  \left( \|f_\alpha\|_{L_2(\Omega)} + \|g\|_{L_2(\Omega)} \right) + N_2 \|f\|_{L_{q_1}(\Omega)},
$$
where $N_1 = N_1(d,\delta, K_1)$ and $N_2 = N_2(d,\delta, q_1, K_1, K)$.
If $q_1=2$, $\Omega = B_R(x_0)$, $x_0 \in \bR^d$, or $\Omega = B_R(x_0) \cap \bR^d_+$, $x_0 \in \partial \bR^d_+$,
then we have
$$
\|Du\|_{L_2(\Omega)} + \|p\|_{L_2(\Omega)} \le N \left( R\|f\|_{L_2(\Omega)} + \|f_\alpha\|_{L_2(\Omega)} + \|g\|_{L_2(\Omega)} \right),
$$
where $N=N(d,\delta)$.
\end{lemma}

\begin{proof}
If $f \equiv 0$, the lemma follows from, for instance, \cite[Lemma 3.1]{CL15}.
To have a non-zero $f$, we obtain $w \in W_{q_1}^2(B_R)$ satisfying the Laplace equation $\Delta w = f 1_{\Omega}$ in $B_R \supset \Omega$ with the zero boundary condition on $\partial B_R$. Then, we solve the above Stokes system by replacing the right-hand side $f + D_\alpha f_\alpha$ by $D_\alpha( D_\alpha w + f_\alpha)$, which makes $N_2$ depend on $K$ and $q_1$.
\end{proof}

The main objective of this section is to prove a reverse H\"older's inequality for $Du$ and $p$.

\begin{lemma}
							\label{lem0712_1}
Let $q_1\in (1,2)$ with $q_1\ge 2d/(d+2)$, and $\Omega \subset \bR^d$ be a bounded domain with $\operatorname{diam} \Omega \le K$ satisfying Assumption \ref{assum1004} with a constant $K_1$ in the estimate \eqref{eq0711_00}.
Suppose that $(u,p) \in W_2^1(\Omega)^d \times L_2(\Omega)$ satisfies
\begin{equation}
							\label{eq0225_01}
\begin{cases}
\cL u + \nabla p = D_\alpha f_\alpha
\quad
&\text{in}\,\,\Omega,
\\
\operatorname{div} u = g
\quad
&\text{in}\,\,\Omega,
\quad
\\
u=0 \quad
&\text{on}\,\,\partial\Omega,
\end{cases}
\end{equation}
where $f_\alpha, g \in L_2(\Omega)$.
Then we have the following.
\begin{enumerate}
\item[(i)] For $x_0 \in \partial\Omega$ and $r \in (0, R_0]$,
\begin{equation}
							\label{eq0707_01}
\begin{aligned}
&(|Du|^2)_{\Omega_{r/2}(x_0)}^{1/2}
+(|p|^2)_{\Omega_{r/2}(x_0)}^{1/2}
\\
&\le N(|f_\alpha|^2)_{\Omega_{r}(x_0)}^{1/2}
+N(|g|^2)_{\Omega_{r}(x_0)}^{1/2}
+N(|Du|^{q_1})_{\Omega_{r}(x_0)}^{1/q_1}
+N(|p|^{q_1})_{\Omega_{r}(x_0)}^{1/{q_1}},
\end{aligned}
\end{equation}
where $N=N(d,\delta, q_1, K_1)$.

\item[(ii)] The same estimate as in \eqref{eq0707_01} holds for $x_0 \in \Omega$ if $r > 0$ and $r \le  \operatorname{dist}(x_0,\Omega)$.
Recall that in this case $\Omega_r(x_0) = B_r(x_0)$.
\end{enumerate}
\end{lemma}

\begin{remark}
							\label{rem0118_1}
Since Assumption \ref{assum1004} is enforced on $\partial\Omega$, as noted in Remarks \ref{rem0918_1} and \ref{rem0229_2}, one can write $d$, $R_0$, and $\operatorname{diam} \Omega$ instead of $K_1$ in the dependency statements of the constants $N$ in Lemma \ref{lem0712_1}.
However, we keep $K_1$ to show that it comes from the estimate of the divergence equation.
\end{remark}

\begin{proof}[Proof of Lemma \ref{lem0712_1}]
Take a cutoff $\eta\in C_0^\infty(B_1)$ satisfying $\eta=1$ on $B_{1/2}$. Denote $\eta_r(x)=\eta(x/r)$.

{\em Case (i)}.
For any $x_0 \in \partial \Omega$ and $r \in (0, R_0]$,
set
$$
\hat u:=\eta_r(\cdot - x_0) u,\quad \hat p:=\eta_r(\cdot - x_0) p.
$$
Without loss of generality, we assume that $x_0 = 0$.
Then, it follows from \eqref{eq0225_01} that $(\hat{u}, \hat{p})$ satisfies
\begin{equation}
							\label{eq0628_02}
\begin{cases}
\cL \hat u + \nabla \hat p = D_\alpha \hat f_\alpha+\hat f
\quad
&\text{in}\,\,\Omega,
\\
\operatorname{div} \hat u = \eta_r g + \nabla \eta_r \cdot u
\quad
&\text{in}\,\,\Omega,
\\
\hat u = 0
\quad
&\text{on} \,\, \partial \Omega,
\end{cases}
\end{equation}
where
\begin{align*}
\hat f_\alpha&=\eta_r f_\alpha+A^{\alpha\beta}uD_\beta \eta_r,\\
\hat f&=-D_\alpha\eta_r f_\alpha+A^{\alpha\beta}D_\beta u D_\alpha\eta_r + \nabla \eta_r p.
\end{align*}
We can extend $\hat f$ to $B_r$, so that $(\hat f)_{B_r}=0$ and $\|\hat f\|_{L_2(B_r)}$ is comparable to $\|\hat f\|_{L_2(\Omega_r)}$.
Indeed, by Assumption \ref{assum1004}, we have
$$
N|B_r| \le |B_r\setminus\Omega_r| \le |B_r|,
$$
where $N=N(d) \in (0,1)$.
Now we set
$$
c := - \frac{1}{|B_r\setminus \Omega_r|} \int_{\Omega_r} \hat{f}(x) \, dx
$$
and define the extension of $\hat{f}$, still denoted by $\hat{f}$, to be
$$
\hat{f}(x) =
\left\{
\begin{aligned}
\hat{f}(x)
\quad &\text{if} \quad x \in \Omega_r,
\\
c \quad &\text{if} \quad x \in B_r\setminus \Omega_r.
\end{aligned}
\right.
$$
Then, $(\hat{f})_{B_r} = 0$ and
$$
\|\hat{f}\|_{L_2(B_r)} \le N(d) \|\hat{f}\|_{L_2(\Omega_r)}.
$$

Let $w^i = (w^i_\alpha)_{\alpha=1}^d = (w^i_1, \ldots, w^i_d) \in \mathring W^1_{q_1}(B_r)^d$, $i=1,\ldots,d$, be a solution (see Remark \ref{rem0229_1}) to
$$
\operatorname{div} w^i = \hat f^i\quad \text{in}\,\,B_r
$$
satisfying
\begin{equation}
                            \label{eq0628_01}
\|Dw^i\|_{L_{q_1}(B_r)}\le N(d,q_1)\|\hat f^i\|_{L_{q_1}(B_r)}
\le N(d,q_1)\|\hat f^i\|_{L_{q_1}(\Omega_r)}.
\end{equation}
By the Sobolev-Poincar\'e inequality and \eqref{eq0628_01}, we get
\begin{align}
                                            \label{eq0628_03}
\|w^i\|_{L_2(B_r)}&\le Nr^{1-d(1/{q_1}-1/2)}\|Dw^i\|_{L_{q_1}(B_r)}\nonumber\\
&\le N(d,q_1)r^{1-d(1/{q_1}-1/2)}\|\hat f^i\|_{L_{q_1}(\Omega_r)}.
\end{align}
We extend $w^i$, $i=1,\ldots,d$, to be zero outside $B_r$.
Then, it follows from \eqref{eq0628_02} that $(\hat{u}, \hat{p} - (\hat{p})_\Omega)$ satisfies
\begin{equation}
							\label{eq0711_01}
\begin{cases}
\cL \hat u + \nabla (\hat p - (\hat{p})_\Omega) = D_\alpha (\hat f_\alpha+ w_\alpha)
\quad
&\text{in}\,\,\Omega,
\\
\operatorname{div} \hat u = \eta_r g + \nabla \eta_r \cdot u
\quad
&\text{in}\,\,\Omega,
\\
\hat u = 0
\quad
&\text{on} \,\, \partial \Omega,
\end{cases}
\end{equation}
where $w_\alpha = (w^1_\alpha, \ldots, w^d_\alpha)^{\operatorname{tr}}$, $\alpha = 1, \ldots, d$.
Using integration by parts, we check that
$$
\int_\Omega \eta_r g+\nabla \eta_r\cdot u=0.
$$
Now we apply the $W^1_2$-estimate in Lemma \ref{lem0225_1} to \eqref{eq0711_01} and use \eqref{eq0628_03} to get
\begin{align*}
&\|D\hat u\|_{L_2(\Omega_r)}+\|\hat p-(\hat p)_\Omega\|_{L_2(\Omega_r)}
\le N_1 \|\hat f_\alpha\|_{L_2(\Omega_r)}\\
&\quad+
N_2 \, r^{1-d(1/{q_1}-1/2)}\|\hat f\|_{L_{q_1}(\Omega_r)}+
N_1\|\eta_r g-\nabla \eta_r\cdot u\|_{L_2(\Omega_r)},
\end{align*}
where $N_1 = N(d,\delta,K_1)$ and $N_2=N(d,\delta, q_1, K_1)$, which implies that
\begin{align*}
&\|D\hat u\|_{L_2(\Omega_r)}+\|\hat p\|_{L_2(\Omega_r)}
\le N_1 \|\hat f_\alpha\|_{L_2(\Omega_r)}+
N_2 \, r^{1-d(1/{q_1}-1/2)}\|\hat f\|_{L_{q_1}(\Omega_r)}
\\
&\quad+ N_1\|\eta_r g-\nabla \eta_r\cdot u\|_{L_2(\Omega_r)}+\| p\|_{L_1(\Omega_r)}|\Omega|^{-1}|\Omega_r|^{1/2}.
\end{align*}
Since $u$ vanishes on $\partial \Omega$, by the boundary Sobolev-Poincar\'e inequality,
$$
r^{-1}(|u|^2)_{\Omega_r}^{1/2}\le N(d) (|Du|^{q_1})_{\Omega_r}^{1/{q_1}}.
$$
Combining the above two inequalities and using H\"older's inequality, we get
\begin{align*}
&(|Du|^2)_{\Omega_{r/2}}^{1/2}
+(|p|^2)_{\Omega_{r/2}}^{1/2}\nonumber\\
&\le N_1(|f_\alpha|^2)_{\Omega_{r}}^{1/2}
+N_1(|g|^2)_{\Omega_{r}}^{1/2}
+N_1 r^{-1}(|u|^2)_{\Omega_{r}}^{1/2}\nonumber\\
&+ N_2(|f_\alpha|^{q_1})_{\Omega_{r}}^{1/{q_1}} +N_2 (|Du|^{q_1})_{\Omega_{r}}^{1/{q_1}}
+N_2 (|p|^{q_1})_{\Omega_{r}}^{1/{q_1}}
+ \frac{|\Omega_r|}{|\Omega|} (|p|^{q_1})_{\Omega_{r}}^{1/{q_1}} \nonumber\\
&\le N(|f_\alpha|^2)_{\Omega_{r}}^{1/2}
+N(|g|^2)_{\Omega_{r}}^{1/2}
+N(|Du|^{q_1})_{\Omega_{r}}^{1/{q_1}}
+N(|p|^{q_1})_{\Omega_{r}}^{1/{q_1}},
\end{align*}
where the last $N$ depends only on $d$, $\delta$, $q_1$, and $K_1$.
This proves \eqref{eq0707_01}.

{\em Case (ii).} For $x_0 \in \Omega$ and $r > 0$ such that $B_r(x_0)\subset \Omega$, we set
$$
\hat u:=\eta_r(\cdot-x_0)\big(u-(u)_{B_r(x_0)}\big),
\quad \hat p:=\eta_r(\cdot-x_0)p.
$$
Again, we assume that $x_0 = 0$.
Then, $(\hat{u}, \hat{p})$ satisfies
\begin{equation}
							\label{eq0711_02}
\begin{cases}
\cL \hat u + \nabla \hat p = D_\alpha \hat f_\alpha+\hat f
\quad
&\text{in}\,\,\Omega,
\\
\operatorname{div} \hat u = \eta_r g + \nabla \eta_r \cdot \left(u - (u)_{B_r}\right)
\quad
&\text{in}\,\,\Omega,
\\
\hat u = 0
\quad
&\text{on} \,\, \partial \Omega,
\end{cases}
\end{equation}
where
\begin{align*}
\hat f_\alpha&=\eta_r f_\alpha+A^{\alpha\beta}(u-(u)_{B_r}\big)
D_\beta \eta_r,\\
\hat f&=-(D_\alpha\eta_r)f_\alpha+A_{\alpha\beta}D_\beta u D_\alpha\eta_r+\nabla \eta_r p.
\end{align*}
Note that $(\hat f)_{B_r}=0$ by the weak formulation of the equation and the fact that $\eta_r$ is supported on $B_r\subset \Omega$.

Let $w^i = (w^i_\alpha)_{\alpha=1}^d = (w^i_1, \ldots, w^i_d) \in \mathring W^1_{q_1}(B_r)^d$, $i=1,\ldots,d$, be a solution (see Remark \ref{rem0229_1}) to
$$
\operatorname{div} w^i = \hat f^i \quad \text{in}\,\,B_r
$$
satisfying
\begin{equation}
                            \label{eq0711_03}
\|Dw^i\|_{L_{q_1}(B_r)}\le N(d,q_1)\|\hat f^i\|_{L_{q_1}(B_r)}.
\end{equation}
By the Sobolev-Poincar\'e inequality and \eqref{eq0711_03}, we get
\begin{align}
                                            \label{eq0711_04}
\|w^i\|_{L_2(B_r)}&\le Nr^{1-d(1/{q_1}-1/2)}\|Dw^i\|_{L_{q_1}(B_r)}\nonumber\\
&\le N(d,{q_1})r^{1-d(1/{q_1}-1/2)}\|\hat f^i\|_{L_{q_1}(B_r)}.
\end{align}
We extend $w^i$, $i=1,\ldots,d$, to be zero outside $B_r$.
Then, it follows from \eqref{eq0711_02} that $(\hat{u}, \hat{p} - (\hat{p})_\Omega)$ satisfies
\begin{equation}
							\label{eq0711_05}
\begin{cases}
\cL \hat u + \nabla (\hat p - (\hat{p})_\Omega) = D_\alpha (\hat f_\alpha+ w_\alpha)
\quad
&\text{in}\,\,\Omega,
\\
\operatorname{div} \hat u = \eta_r g + \nabla \eta_r \cdot \left(u - (u)_{B_r}\right)
\quad
&\text{in}\,\,\Omega,
\\
\hat u = 0
\quad
&\text{on} \,\, \partial \Omega,
\end{cases}
\end{equation}
where $w_\alpha = (w^1_\alpha, \ldots, w^d_\alpha)^{\operatorname{tr}}$, $\alpha = 1, \ldots, d$.
Now, we apply the $W^1_2$-estimate in Lemma \ref{lem0225_1} to \eqref{eq0711_05} and use \eqref{eq0711_04} to get
\begin{align*}
&\|D\hat u\|_{L_2(B_r)}+\|\hat p-(\hat p)_\Omega\|_{L_2(B_r)}
\le N_1 \|\hat f_\alpha\|_{L_2(B_r)}\\
&\quad+
N_2 \, r^{1-d(1/{q_1}-1/2)}\|\hat f\|_{L_{q_1}(B_r)}+
N_1\|\eta_r g-\nabla \eta_r\cdot \left(u - (u)_{B_r}\right)\|_{L_2(B_r)},
\end{align*}
where $N_1 = N(d,\delta,K_1)$ and $N_2=N(d,\delta, q_1, K_1)$. This implies that
\begin{align*}
&\|D\hat u\|_{L_2(B_r)}+\|\hat p\|_{L_2(B_r)}
\le N_1 \|\hat f_\alpha\|_{L_2(B_r)}+
N_2 \, r^{1-d(1/{q_1}-1/2)}\|\hat f\|_{L_{q_1}(B_r)}
\\
&\quad+ N_1\|\eta_r g-\nabla \eta_r\cdot \left(u - (u)_{B_r}\right)\|_{L_2(B_r)}+\| p\|_{L_1(B_r)}|\Omega|^{-1}|B_r|^{1/2}.
\end{align*}
The Sobolev-Poincar\'e inequality shows that
$$
r^{-1}(|u-(u)_{B_r}|^2)_{B_r}^{1/2}\le N(d) (|Du|^{q_1})_{B_r}^{1/{q_1}}.
$$
Combining the above two inequalities and using H\"older's inequality, we get
\begin{align*}
&(|Du|^2)_{B_{r/2}}^{1/2}
+(|p|^2)_{B_{r/2}}^{1/2}\nonumber\\
&\le N_1(|f_\alpha|^2)_{B_r}^{1/2}
+N_1(|g|^2)_{B_r}^{1/2}
+N_1 r^{-1}(|u - (u)_{B_r}|^2)_{B_r}^{1/2}\nonumber\\
&\quad + N_2(|f_\alpha|^{q_1})_{B_r}^{1/{q_1}} +N_2 (|Du|^{q_1})_{B_r}^{1/{q_1}}
+N_2 (|p|^{q_1})_{B_r}^{1/{q_1}}
+ \frac{|B_r|}{|\Omega|} (|p|^{q_1})_{B_r}^{1/{q_1}} \nonumber\\
&\le N(|f_\alpha|^2)_{B_r}^{1/2}
+N(|g|^2)_{B_r}^{1/2}
+N(|Du|^{q_1})_{B_r}^{1/{q_1}}
+N(|p|^{q_1})_{B_r}^{1/{q_1}},
\end{align*}
where the last $N$ depends only on $d$, $\delta$, $q_1$, and $K_1$.
This proves \eqref{eq0707_01} when $B_r(x_0) \subset \Omega$.
\end{proof}

The following proposition can be found, for instance, in \cite[Ch. V]{MR717034}.

\begin{proposition}
							\label{prop0713_1}
Let $1< q_0 < q_1$, $\Phi \ge 0$ in a $d$-dimensional cube $Q$, and $\Psi \in L_{q_1}(Q)$.
Suppose that
$$
\dashint_{B_r(x_0)} \Phi^{q_0} \, dx \le N_0 \left( \dashint_{B_{8r}(x_0)} \Phi \right)^{q_0} + N_0 \dashint_{B_{8r}(x_0)} \Psi^{q_0} \, dx + \theta \dashint_{B_{8r}(x_0)} \Phi^{q_0} \, dx
$$
for every $x_0 \in Q$ and $r \in (0,R_2]$ such that $B_{8r}(x_0) \subset Q$, where $R_2$ and $\theta$ are constants with $R_2 > 0$ and $\theta \in [0,1)$.
Then, $\Phi \in L_{\tilde{q},\operatorname{loc}}(Q)$ for $\tilde{q} \in [q_0, q_0+\varepsilon)$ and
$$
\left( \dashint_{B_r(x_0)} \Phi^{\tilde{q}} \, dx \right)^{1/\tilde{q}} \le N \left( \dashint_{B_{8r}(x_0)} \Phi^{q_0} \right)^{1/q_0} + N \left(\dashint_{B_{8r}(x_0)} \Psi^{\tilde{q}} \, dx \right)^{1/\tilde{q}}
$$
for all $B_{8r}(x_0) \subset Q$, $r < R_2$,
where $N$ and $\varepsilon$ depend only on $d$, $q_0$, $q_1$, $\theta$, and $N_0$, and $\varepsilon$ satisfies $0 < \varepsilon < q_1 - q_0$.
\end{proposition}

\begin{lemma}[Reverse H\"{o}lder's inequality]
							\label{lem0715_1}
Let $\Omega \subset \bR^d$ be a bounded domain with $\operatorname{diam} \Omega \le K$ satisfying Assumption \ref{assum1004} with a constant $K_1$ in the estimate \eqref{eq0711_00}.
Suppose that $(u,p) \in \mathring{W}_2^1(\Omega)^d \times L_2(\Omega)$ satisfies \eqref{eq0225_01},
where $f_\alpha, g \in L_{q_1}(\Omega)$, $q_1 > 2$, and $(g)_\Omega=0$.
Then, there exist constants $\tilde{q} \in (2,q_1)$ and $N>0$, depending only on $d$, $\delta$, $K_1$, and $q_1$, such that
\begin{multline*}
\left(|D \bar{u}|^{\tilde{q}} \right)^{1/\tilde{q}}_{B_r(x_0)}
+ \left(|\bar{p}|^{\tilde{q}} \right)^{1/\tilde{q}}_{B_r(x_0)}
\\
\le N\left(|D \bar{u}|^2 \right)^{1/2}_{B_{8r}(x_0)} + N \left(| \bar{p}|^2 \right)^{1/2}_{B_{8r}(x_0)}
+ N \left(|\bar{f}_\alpha|^{\tilde{q}} \right)^{1/\tilde{q}}_{B_{8r}(x_0)} + \left(|\bar{g}|^{\tilde{q}} \right)^{1/\tilde{q}}_{B_{8r}(x_0)}
\end{multline*}
for any $x_0 \in \bR^d$ and $r \in (0, R_0/8]$, where
$\bar{u}$, $\bar{p}$, $\bar{f}_\alpha$, and $\bar{g}$ are the extensions of $u$, $p$, $f_\alpha$, and $g$ to $\bR^d$ so that they are zero on $\bR^d \setminus \Omega$.
\end{lemma}

\begin{proof}
We fix a constant $q_1\in ( 2d/(d+2),2)$.
We first prove that
\begin{multline}
							\label{eq0713_01}
(|D\bar{u}|^2)_{B_{r/2}(x_0)}^{1/2}
+(|\bar{p}|^2)_{B_{r/2}(x_0)}^{1/2}
\\
\le N(|\bar{f}_\alpha|^2)_{B_{4r}(x_0)}^{1/2} + N(|\bar{g}|^2)_{B_{4r}(x_0)}^{1/2} + N(|D\bar{u}|^{q_1})_{B_{4r}(x_0)}^{1/{q_1}} + N(|\bar{p}|^{q_1})_{B_{4r}(x_0)}^{1/{q_1}}	
\end{multline}
for any $x_0 \in \bR^d$ and $r \in (0,R_0/4]$, where $N=N(d,\delta, K_1)$.
Once this is proved, we apply Proposition \ref{prop0713_1} to obtain the desired estimate.

To prove \eqref{eq0713_01}, we
consider the following three cases:
$$
B_r(x_0) \subset \Omega,
\quad
B_r(x_0) \cap \partial \Omega \neq \emptyset,
\quad
B_r(x_0) \subset \bR^d \setminus \Omega.
$$

{\em Case 1}: $B_r(x_0) \subset \Omega$. In this case, by the second assertion of Lemma \ref{lem0712_1}, we have
\begin{equation*}
\begin{aligned}
&(|D\bar{u}|^2)_{B_{r/2}(x_0)}^{1/2}
+(|\bar{p}|^2)_{B_{r/2}(x_0)}^{1/2} = (|D u|^2)_{B_{r/2}(x_0)}^{1/2}
+(|p|^2)_{B_{r/2}(x_0)}^{1/2}
\\
&\le N(|f_\alpha|^2)_{B_r(x_0)}^{1/2} + N(|g|^2)_{B_r(x_0)}^{1/2} + N(|Du|^{q_1})_{B_r(x_0)}^{1/{q_1}} + N(|p|^{q_1})_{B_r(x_0)}^{1/{q_1}}
\\
&\le N(|\bar{f}_\alpha|^2)_{B_{4r}(x_0)}^{1/2} + N(|\bar{g}|^2)_{B_{4r}(x_0)}^{1/2} + N(|D\bar{u}|^{q_1})_{B_{4r}(x_0)}^{1/{q_1}} + N(|\bar{p}|^{q_1})_{B_{4r}(x_0)}^{1/{q_1}},
\end{aligned}
\end{equation*}
where $N=N(d,\delta, q_1, K_1)$.

{\em Case 2}: $B_r(x_0) \cap \partial \Omega \neq \emptyset$.
Take $y_0 \in  \partial \Omega$ such that $|x_0 - y_0| = \operatorname{dist}(x_0, \partial\Omega) \le r$.
Then, we see that
$$
B_{r/2}(x_0) \subset B_{3r/2}(y_0) \subset B_{3r}(y_0) \subset B_{4r}(x_0).
$$
Recall that  $r \le R_0/8$.
By the first assertion of Lemma \ref{lem0712_1} and the fact that $|B_{3r}(y_0)| \le N(d) |\Omega_{3r}(y_0)|$, we have
\begin{equation*}
\begin{aligned}
&(|D\bar{u}|^2)_{B_{r/2}(x_0)}^{1/2}
+(|\bar{p}|^2)_{B_{r/2}(x_0)}^{1/2} \le N (|D \bar{u}|^2)_{B_{3r/2}(y_0)}^{1/2}
+N (|\bar{p}|^2)_{B_{3r/2}(y_0)}^{1/2}
\\
&\le N (|D u|^2)_{\Omega_{3r/2}(y_0)}^{1/2}
+N (|p|^2)_{\Omega_{3r/2}(y_0)}^{1/2}
\\
&\le N(|f_\alpha|^2)_{\Omega_{3r}(y_0)}^{1/2} + N(|g|^2)_{\Omega_{3r}(y_0)}^{1/2} + N(|Du|^{q_1})_{\Omega_{3r}(y_0)}^{1/{q_1}} + N(|p|^{q_1})_{\Omega_{3r}(y_0)}^{1/{q_1}}
\\
&\le N(|\bar{f}_\alpha|^2)_{B_{4r}(x_0)}^{1/2} + N(|\bar{g}|^2)_{B_{4r}(x_0)}^{1/2} + N(|D\bar{u}|^{q_1})_{B_{4r}(x_0)}^{1/{q_1}} + N(|\bar{p}|^{q_1})_{B_{4r}(x_0)}^{1/{q_1}},
\end{aligned}
\end{equation*}
where $N=N(d,\delta, K_1)$.

{\em Case 3}: $B_r(x_0) \subset \bR^d \setminus \Omega$.
By the definitions of $\bar{u}$ and $\bar{p}$, the inequality \eqref{eq0713_01} trivially holds.
\end{proof}

\section{$L_\infty$ and H\"{o}lder estimates}
							\label{sec04}

In this section, we prove $L_\infty$ and H\"older estimates of certain linear combinations of $Du$ and $p$, which are crucial in proving our main results.
We set
\begin{equation*}
\cL_0 u = D_\alpha \left( A^{\alpha\beta}(x_1) D_\beta u \right),
\end{equation*}
where $A^{\alpha\beta}(x_1) = [A^{\alpha\beta}_{ij}(x_1)]_{i,j=1}^d$ are functions of only $x_1 \in \bR$.
Note that we do not impose any regularity assumptions on $A^{\alpha\beta}(x_1)$.
In this case, if a sufficiently smooth $(u,p)$ satisfies $\cL_0 u + \nabla p = 0$ in $\Omega \subset \bR^d$, we see that
\begin{equation}
							\label{eq0127_01}
D_1 \left( A^{1\beta} D_\beta u +
\left[
\begin{matrix}
p
\\
0
\\
\vdots
\\
0
\end{matrix}
\right]
\right)
= - \sum_{\alpha \ne 1} A^{\alpha\beta} D_{\alpha\beta} u -
\left[
\begin{matrix}
0
\\
D_2 p
\\
\vdots
\\
D_d p
\end{matrix}
\right]
\end{equation}
in $\Omega$.
Set $U = \left(U_1, U_2, \ldots, U_d\right)^{\operatorname{tr}}$, where
\begin{equation}
							\label{eq0829_01}
U_1 = \sum_{j=1}^d \sum_{\beta=1}^d A_{1j}^{1\beta} D_\beta u_j + p,
\quad
U_i = \sum_{j=1}^d \sum_{\beta=1}^d A_{ij}^{1\beta} D_\beta u_j,
\quad
i = 2, \ldots, d.
\end{equation}
That is,
$$
U = A^{1\beta}D_\beta u + (p,0,\ldots,0)^{\operatorname{tr}}.
$$

Throughout the paper, we write $D D_{x'}^k u$, $k=0,1, \ldots$, to denote $D^\vartheta u$,
where $\vartheta$ is a multi-index such that $\vartheta = (\vartheta_1, \ldots, \vartheta_d)$ with $\vartheta_1 = 0, 1$ and $|\vartheta| = k+1$.
As usual, we denote by $C^\tau(\Omega)$, $\tau \in (0,1)$, the H\"{o}lder space equipped with the norm
$$
\|u\|_{C^{\tau}(\Omega)} = \|u\|_{L_\infty(\Omega)} + [u]_{C^{\tau}(\Omega)},
$$
where $[u]_{C^\tau(\Omega)}$ is the H\"{o}lder semi-norm of $u$ defined by
$$
[u]_{C^\tau(\Omega)} = \sup_{\substack{x,y \in \Omega \\ x \ne y}} \frac{|u(x) - u(y)|}{|x-y|^\tau}.
$$

First, we prove the following lemma, a version of which was proved in \cite{DK15S} for the case $r = 1$ and $R=2$.
The objective is to ensure that the right-hand sides of the inequalities depend only on $(R-r)$, instead of $r$ or $R$.

\begin{lemma}
							\label{lem0829_2}
Let $0 < r < R$ and let $\ell$ be a constant.
\begin{enumerate}
\item[(i)] If $(u,p) \in W_2^1(B_R)^d \times L_2(B_R)$ satisfies
\begin{equation}
							\label{eq0229_01}
\left\{
\begin{aligned}
 \cL_0 u + \nabla p &= 0
\quad
\text{in}\,\,B_R,
\\
\operatorname{div} u &= \ell
\quad
\text{in}\,\,B_R,
\end{aligned}
\right.
\end{equation}
then 
\begin{equation}
							\label{eq0905_06}
\|Du\|_{L_\infty(B_r)}
\le N (R-r)^{-d/2}\|Du\|_{L_2(B_R)},	
\end{equation}
\begin{equation}
							\label{eq0905_08}
\|p\|_{L_\infty(B_r)}
\\
\le N (R-r)^{-d/2}\|Du\|_{L_2(B_R)} + N (R-r)^{-d/2}\|p\|_{L_2(B_R)},
\end{equation}
\begin{equation}
							\label{eq0902_04}
\left[D_{x'}u \right]_{C^{1/2}(B_r)} + \left[U\right]_{C^{1/2}(B_r)}
\le N (R-r)^{-(d+1)/2} \|Du\|_{L_2(B_R)},
\end{equation}
where $N=N(d,\delta)$.

\item[(ii)] If $(u,p) \in W_2^1(B_R^+)^d \times L_2(B_R^+)$ satisfies
\begin{equation}
							\label{eq0229_02}
\left\{
\begin{aligned}
\cL_0 u + \nabla p &= 0
\quad
\text{in}\,\,B_R^+,
\\
\operatorname{div} u &= \ell
\quad
\text{in}\,\,B_R^+,
\\
u &= 0
\quad
\text{on}\,\, B_R \cap \partial \bR^d_+,
\end{aligned}
\right.
\end{equation}
then 
\begin{equation}
							\label{eq0905_02}
\|Du\|_{L_\infty(B_r^+)}
\le N (R-r)^{-d/2}\|Du\|_{L_2(B_R^+)},
\end{equation}
\begin{equation}
							\label{eq0905_01}
\|p\|_{L_\infty(B_r^+)} \le N (R-r)^{-d/2}\|Du\|_{L_2(B_R^+)} + N (R-r)^{-d/2} \|p\|_{L_2(B_R^+)},
\end{equation}
\begin{equation}
							\label{eq0902_05}
\left[D_{x'}u \right]_{C^{1/2}(B_r^+)} + \left[U\right]_{C^{1/2}(B_r^+)}
\le N (R-r)^{-(d+1)/2} \|Du\|_{L_2(B_R^+)},
\end{equation}
where $N=N(d,\delta)$.
\end{enumerate}
\end{lemma}

\begin{proof}
We first prove \eqref{eq0905_06} and \eqref{eq0905_08} when $r=1$ and $R=2$.
By Lemma 4.3 of \cite{DK15S}, we have
\begin{equation}
                        \label{eq0905_03}
\|D_{x'}u\|_{L_\infty(B_1)}+\sum_{i=2}^d\|U_i\|_{L_\infty(B_1)}
+[U_1]_{C^{1/2}(B_1)}
\le N \|Du\|_{L_2(B_2)},
\end{equation}
where $N=N(d,\delta)$.
Then using the relation
$$
\operatorname{div} u = \ell,
$$
we obtain 
\begin{align}
                \label{eq0905_04}
&\|D_1u_1\|_{L_\infty(B_1)} \le |\ell| + N\|Du\|_{L_2(B_2)}\nonumber\\
&\le N\|\operatorname{div} u\|_{L_2(B_2)}+N\|Du\|_{L_2(B_2)}\le N\|Du\|_{L_2(B_2)}.
\end{align}
From \eqref{eq0829_01}, we have
\begin{equation*}
\sum_{j=2}^d A_{ij}^{11} D_1 u_j=U_i-\sum_{j=1}^d \sum_{\beta=2}^dA_{ij}^{1\beta} D_\beta u_j-A^{11}_{i1}D_1 u_1,
\quad
i=2,\ldots,d.
\end{equation*}
By the ellipticity condition, $[A^{11}_{ij}]_{i,j=2}^d$ is nondegenerate. Hence, using \eqref{eq0905_03} and \eqref{eq0905_04}, we deduce that
\begin{equation}
                \label{eq0905_05}
\|D_1u_j\|_{L_\infty(B_1)} \le N\|Du\|_{L_2(B_2)}
\end{equation}
for $j=2,\ldots,d$.
To estimate the $L_\infty$-norm of $p$, we use the interpolation inequality
$$
\|U_1\|_{L_\infty(B_1)}\le N\|U_1\|_{L_2(B_1)}+N[U_1]_{C^{1/2}(B_1)},
$$
the definition of $U_1$, and \eqref{eq0905_03} to obtain
$$
\|U_1\|_{L_\infty(B_1)}\le N\|Du\|_{L_2(B_2)}+N\|p\|_{L_2(B_1)}.
$$
This inequality, along with the definition of $U_1$ and \eqref{eq0905_03}, \eqref{eq0905_04}, and \eqref{eq0905_05} gives
$$
\|p\|_{L_\infty(B_1)}\le N\|Du\|_{L_2(B_2)}+N\|p\|_{L_2(B_1)}.
$$
From this, \eqref{eq0905_03}, \eqref{eq0905_04}, and \eqref{eq0905_05}, we conclude that \eqref{eq0905_06} and \eqref{eq0905_08} are satisfied when $r=1$ and $R=2$.

To prove \eqref{eq0905_06} for general $0 < r < R$, we use a scaling argument.
For $x \in B_r$, we see that $B_{(R-r)/2}(x) \subset B_{R-r}(x) \subset B_R$.
Then, by using the following scaling with a translation
\begin{equation}
							\label{eq0905_07}
\hat{u}(y) = 2(R-r)^{-1} u\left( \frac{R-r}{2}y + x \right), \quad \hat{p}(y) = p\left( \frac{R-r}{2}y + x\right)
\end{equation}
and \eqref{eq0905_06} for $r=1$ and $R=2$, we have
\begin{multline*}
\|Du\|_{L_\infty(B_{(R-r)/2}(x))}
\le N (R-r)^{-d/2}\|Du\|_{L_2(B_{R-r}(x))}
\\
\le N (R-r)^{-d/2}\|Du\|_{L_2(B_R)},
\end{multline*}
which implies \eqref{eq0905_06} for general $0<r<R$.
Similarly, we obtain \eqref{eq0905_08} for general $0<r<R$.

For the proof of \eqref{eq0902_04}, see below that of \eqref{eq0902_05}.

Now, we consider \eqref{eq0905_02} and \eqref{eq0905_01}.
By using Lemma 4.4 in \cite{DK15S} and following the same steps as in the proof of \eqref{eq0905_06} and \eqref{eq0905_08} for the case $r=1$ and $R=2$, we obtain \eqref{eq0905_02} and \eqref{eq0905_01} for $r=1$ and $R=2$.
For general $0<r<R$, let $x = (x_1,x') \in B_r^+$.
We consider two cases: $x_1 \ge (R-r)/2$ and $x_1 < (R-r)/2$.

If $x_1 \ge (R-r)/2$, then $B_{(R-r)/2}(x) \subset B_R^+$.
By \eqref{eq0905_06} we have
\begin{multline}
							\label{eq0902_02}
\|D u\|_{L_\infty(B_{(R-r)/4}(x))}
\le N(R-r)^{-d/2}\|Du\|_{L_2(B_{(R-r)/2}(x))}
\\
\le N(R-r)^{-d/2}\|Du\|_{L_2(B_R^+)}.
\end{multline}

If $x_1 < (R-r)/2$, it follows that
$$
x \in B_{(R-r)/2}^+(0,x') \subset B_{R-r}^+(0,x') \subset B_R^+.
$$
Then, we use \eqref{eq0905_02} for $r=1$ and $R=2$ and a scaling as in \eqref{eq0905_07} to obtain
\begin{multline}
							\label{eq0902_03}
\|Du\|_{L_\infty(B_{(R-r)/2}^+(0,x'))}
\le N(R-r)^{-d/2}\|Du\|_{L_2(B_{R-r}^+(0,x'))}
\\
\le N(R-r)^{-d/2}\|Du\|_{L_2(B_R^+)}.
\end{multline}
From \eqref{eq0902_02} and \eqref{eq0902_03}, we arrive at \eqref{eq0905_02}.
We similarly prove \eqref{eq0905_01} for general $0<r<R$.

To prove \eqref{eq0902_05}, let $x,y \in B_r^+$.
We consider two cases: $|x-y| \ge (R-r)/8$ and $|x-y| < (R-r)/8$.

If $|x-y| \ge (R-r)/8$, from the estimate \eqref{eq0905_02} it follows that
\begin{equation}
							\label{eq0905_09}
\frac{|D_{x'}u(x) - D_{x'}u(y)|}{|x-y|^{1/2}} + \frac{|U_i(x) - U_i(y)|}{|x-y|^{1/2}}
\le N(R-r)^{-(d+1)/2}\|Du\|_{L_2(B_R^+)},
\end{equation}
where $i=2,\ldots,d$.
For $U_1$, we see that
$$
\frac{|U_1(x)-U_1(y)|}{|x-y|^{1/2}} \le N (R-r)^{-1/2} \|U_1 - c\|_{L_\infty(B_r^+)}
$$
for any constant $c$, where
$$
\|U_1 - c\|_{L_\infty(B_r^+)} \le N\|Du\|_{L_\infty(B_r^+)} + \|p - c\|_{L_\infty(B_r^+)}.
$$
Since the system \eqref{eq0229_02} is satisfied by $(u, p-c)$ in place of $(u,p)$, from \eqref{eq0905_01} with $(u,p-c)$ we have
$$
\|p - c\|_{L_\infty(B_r^+)} \le N(R-r)^{-d/2}\|Du\|_{L_2(B_R^+)} + N(R-r)^{-d/2}\|p-c\|_{L_2(B_R^+)}.
$$
On the other hand, from Lemma 3.5 in \cite{DK15S} with $c = (p)_{B_R^+}$, it follows that
$$
\|p-c\|_{L_2(B_R^+)} \le N\|Du\|_{L_2(B_R^+)},
$$
where $N=N(d,\delta)$.
Combining the last three inequalities with \eqref{eq0905_02}, we obtain 
\begin{equation}
							\label{eq0905_10}
\frac{|U_1(x)-U_1(y)|}{|x-y|^{1/2}} \le N(R-r)^{-(d+1)/2}\|Du\|_{L_2(B_R^+)},
\end{equation}
provided that $|x-y| \ge (R-r)/8$.

If $|x-y| < (R-r)/8$, we consider two sub-cases: $x_1 \ge (R-r)/4$ and $x_1 < (R-r)/4$.
If $x_1 \ge (R-r)/4$, then
$$
y \in B_{(R-r)/8}(x) \subset B_{(R-r)/4}(x) \subset B_R^+.
$$
Thus, by \cite[Lemma 4.3]{DK15S} with a scaling and translation it follows that
\begin{multline}
							\label{eq0905_11}
[D_{x'}u]_{C^{1/2}(B_{(R-r)/8}(x))} + [U]_{C^{1/2}(B_{(R-r)/8}(x))}
\\
\le N(R-r)^{-(d+1)/2}\|Du\|_{L_2(B_{(R-r)/4}(x))} \le N(R-r)^{-(d+1)/2}\|Du\|_{L_2(B_R^+)}.
\end{multline}
If $x_1< (R-r)/4$, then
$$
x,y \in B_{(R-r)/2}^+ (0,x') \subset B_{R-r}^+(0,x') \subset B_R^+.
$$
In this case, we use \cite[Lemma 4.4]{DK15S} with a scaling and translation to get
\begin{multline}
							\label{eq0905_12}
[D_{x'}u]_{C^{1/2}(B_{(R-r)/2}^+(0,x'))} + [U]_{C^{1/2}(B_{(R-r)/2}^+(0,x'))}
\\
\le N(R-r)^{-(d+1)/2}\|Du\|_{L_2(B_{R-r}^+(0,x'))} \le N(R-r)^{-(d+1)/2}\|Du\|_{L_2(B_R^+)}.
\end{multline}
From \eqref{eq0905_09}, \eqref{eq0905_10}, \eqref{eq0905_11}, and \eqref{eq0905_12}, we conclude \eqref{eq0902_05}.
The lemma is proved.
\end{proof}

The lemma below shows that the estimates in Lemma \ref{lem0829_2} also hold under the assumption that $(u,p) \in W_q^1(B_R)^d \times L_q(B_R)$ or $(u,p) \in W_q^1(B_R^+)^d \times L_q(B_R^+)$, where $q \in (1,\infty)$.

\begin{lemma}
							\label{lem0829_3}
Let $0<r<R$ and $1 < q < \infty$, and let $\ell$ be a constant.
\begin{enumerate}
\item[(i)] Let $(u,p) \in W_q^1(B_R)^d \times L_q(B_R)$ satisfy \eqref{eq0229_01}.
Then we have
\begin{align*}
\|D u\|_{L_\infty(B_r)}
&\le N(R-r)^{-d} \|Du\|_{L_1(B_R)},
\\
\|p\|_{L_\infty(B_r)} &\le N(R-r)^{-d}\|Du\|_{L_1(B_R)} + N(R-r)^{-d}\|p\|_{L_1(B_R)},
\end{align*}
$$
\left[D_{x'}u \right]_{C^{1/2}(B_r)} + \left[U\right]_{C^{1/2}(B_r)}
\le N (R-r)^{-d-1/2} \|Du\|_{L_1(B_R)},	
$$
where $N=N(d,\delta)$.

\item[(ii)] Let $(u,p) \in W_q^1(B_R^+)^d \times L_q(B_R^+)$ satisfy \eqref{eq0229_02}.
The same inequalities as in (i) hold with $B_r^+$ and $B_R^+$ in place of $B_r$ and $B_R$, respectively.
\end{enumerate}
\end{lemma}

\begin{proof}
We use an iteration argument. It is clear that we can assume $q \in (1,2)$.
We only deal with the second assertion.
Let
$$
r_0 = (R+2r)/3, \quad r_1  = (2R+r)/3,
$$
and take a non-negative infinitely differentiable function $\varphi$ defined on $\bR^{d-1}$ such that it has a compact support in $B_1' = \{x'\in \bR^{d-1}: |x'| < 1\}$ and $\int_{\bR^{d-1}} \varphi \, dx'= 1$.
For a sufficiently small $\varepsilon \in (0,(R-r)/3)$, set
$$
u^{(\varepsilon)}(x_1,x') = \int_{\bR^{d-1}} u(x_1,x'-y')\varphi_\varepsilon(y') \, dy',
$$
$$
p^{(\varepsilon)}(x_1,x') = \int_{\bR^{d-1}} p(x_1,x'-y')\varphi_\varepsilon(y') \, dy',
$$
where $(x_1,x') \in B_{r_1}^+$ and $\varphi_\varepsilon(y') = \varepsilon^{1-d}\varepsilon(y'/\varepsilon)$.
That is, they are the mollifications of $(u,p)$ with respect to the variables $x' \in \bR^{d-1}$.
Note that
$$
D(u^{(\varepsilon)}) = (Du)^{(\varepsilon)},
\quad
U^{(\varepsilon)} = A^{1\beta} D_\beta u^{(\varepsilon)} + (p^{(\varepsilon)},0,\ldots,0)^{\operatorname{tr}}
$$
in $B_{r_1}^+$.
Since the coefficients of $\cL_0$ are functions of only $x_1$, we see that $(u^{(\varepsilon)}, p^{(\varepsilon)})$ satisfies the equation in \eqref{eq0229_02} with $B_{r_1}^+$ in place of $B_R^+$.
We first claim that
\begin{equation}
							\label{eq0913_02}
(u^{(\varepsilon)}, p^{(\varepsilon)}) \in W_2^1(B_{r_1}^+)^d \times L_2(B_{r_1}^+).
\end{equation}
Note that, for $i=0,1$ and for any non-negative integer $k$,
\begin{equation}
							\label{eq0825_01}
D_1^i D_{x'}^k u^{(\varepsilon)}, \,\, D_{x'}^k p^{(\varepsilon)} \in L_q(B_{r_1}^+).	
\end{equation}
From this and an anisotropic Sobolev embedding theorem (see, for instance, \cite[Lemma 3.5]{MR2800569}) it follows that
\begin{equation}
							\label{eq0825_02}
u^{(\varepsilon)}, \,\, D_{x'} u^{(\varepsilon)} \in L_2(B_{r_1}^+).	
\end{equation}
This together with the relation $\operatorname{div} u^{(\varepsilon)} = \ell$ in $B_{r_1}^+$ shows that
\begin{equation}
							\label{eq0825_03}
D_1 u_1^{(\varepsilon)} \in L_2(B_{r_1}^+).
\end{equation}
Next, by the definition of $U_i$, $i=1,\ldots,d$, we see that, for any non-negative integer $k$,
$$
D_{x'}^k U_i^{(\varepsilon)} \in L_q(B_{r_1}^+),
\quad
i=1,\ldots,d.
$$
From \eqref{eq0825_01} and the relation \eqref{eq0127_01} we also see that for any non-negative integer $k$,
$$
D_1 D_{x'}^k U_i^{(\varepsilon)} \in L_q(B_{r_1}^+),
\quad
i=1,\ldots,d.
$$
Again by the anisotropic Sobolev embedding theorem, we obtain that
\begin{equation}
							\label{eq0913_01}
U^{(\varepsilon)}_i \in L_2(B_{r_1}^+),
\quad
i=1,\ldots,d.
\end{equation}
The definition of $U_i$, $i=2,\ldots,d$, shows that
$$
\sum_{j=2}^d A_{ij}^{11} D_1u_j^{(\varepsilon)} = U_i^{(\varepsilon)} - \sum_{j=1}^d\sum_{\beta=2}^d A_{ij}^{1\beta} D_{\beta}u_j^{(\varepsilon)} - A_{i1}^{11} D_1 u_1^{(\varepsilon)},
\quad
i=2,\ldots, d.
$$
By the ellipticity of $A^{\alpha\beta}$, \eqref{eq0825_02}, \eqref{eq0825_03}, and \eqref{eq0913_01}, we have
$$
D_1u_j^{(\varepsilon)} \in L_2(B_{r_1}^+),
\quad
j=2,\ldots,d.
$$
Finally, from the definition of $U_1$ and the fact that $U_1^{(\varepsilon)}, Du^{(\varepsilon)} \in L_2(B_{r_1}^+)$ shown above, we conclude that
$$
p^{(\varepsilon)} \in L_2(B_{r_1}^+).
$$
Hence, \eqref{eq0913_02} is proved.

Set
$$
\Gamma_0 = r_0,
\quad
\Gamma_m = r_0 + (r_1-r_0)\sum_{k=1}^m\frac{1}{2^k},
\quad
m=1,2,\ldots,
$$
and denote $\Omega_m = B_{\Gamma_m}^+$.
By applying Lemma \ref{lem0829_2} to $(u^{(\varepsilon)}, p^{(\varepsilon)})$ which satisfies \eqref{eq0913_02} and \eqref{eq0229_02} in $\Omega_{m+1}$, we have
$$
\|D u^{(\varepsilon)} \|_{L_\infty(\Omega_m)} \le N_0 \left(\frac{2^m}{r_1-r_0}\right)^{d/2} \|Du^{(\varepsilon)}\|_{L_2(\Omega_{m+1})},
\quad
m=0,1,2,\ldots,
$$
where $N_0$ is the constant in Lemma \ref{lem0829_2} that depends only on $d$ and $\delta$.
On the other hand, for any $ \varepsilon_0 > 0$ we have
\begin{multline*}
\|D u^{(\varepsilon)}\|_{L_2(\Omega_m)}
\le \frac{\varepsilon_0}{N_0}\left(\frac{r_1-r_0}{2^m}\right)^{d/2} \|D u^{(\varepsilon)}\|_{L_\infty(\Omega_m)}
\\
+ \frac{N_0}{4\varepsilon_0} \left(\frac{2^m}{r_1-r_0}\right)^{d/2}\|D u^{(\varepsilon)}\|_{L_1(\Omega_m)}.
\end{multline*}
By combining the above two inequalities, it follows that
\begin{equation*}
\|Du^{(\varepsilon)}\|_{L_2(\Omega_m)} \le
 \varepsilon_0 \|Du^{(\varepsilon)}\|_{L_2(\Omega_{m+1})} + \frac{N}{\varepsilon_0} \left(\frac{2^m}{r_1-r_0}\right)^{d/2} \|Du^{(\varepsilon)}\|_{L_1(\Omega_{m})},
\end{equation*}
where $N=N(d,\delta)$.
We multiply both sides by $\varepsilon_0^m$ and make summations with respect to $m=0,1,\ldots$ to get
\begin{align*}
&\|Du^{(\varepsilon)}\|_{L_2(B_{r_0}^+)} + \sum_{m=1}^\infty \varepsilon_0^m \|Du^{(\varepsilon)}\|_{L_2(\Omega_m)} \\
&\le \sum_{m=1}^\infty \varepsilon_0^m \|Du^{(\varepsilon)}\|_{L_\infty(\Omega_m)}+ N (r_1-r_0)^{-d/2} \|Du^{(\varepsilon)}\|_{L_1(B_{r_1}^+)} \varepsilon_0^{-1} \sum_{m=0}^\infty (\varepsilon_0 2^{d/2})^m.
\end{align*}
Take a sufficiently small $\varepsilon_0 > 0$ so that $\varepsilon_0 2^{d/2} < 1$.
Then we see that the summations above are all finite and, by removing the same terms from both sides of the inequality, we get
\begin{equation}
							\label{eq0914_02}
\|Du^{(\varepsilon)}\|_{L_2(B_{r_0}^+)} \le N(r_1-r_0)^{-d/2} \|Du^{(\varepsilon)}\|_{L_1(B_{r_1}^+)},
\end{equation}
where $N=N(d,\delta)$.
Similarly, we also obtain
\begin{equation*}
\|p^{(\varepsilon)}\|_{L_2(B_{r_0}^+)} \le N(R-r)^{-d/2} \|Du^{(\varepsilon)}\|_{L_1(B_{r_1}^+)}
+ N(R-r)^{-d/2} \|p^{(\varepsilon)}\|_{L_1(B_{r_1}^+)}.
\end{equation*}
This together with \eqref{eq0914_02} and Lemma \ref{lem0829_2} applied to $(u^{(\varepsilon)},p^{(\varepsilon)}$) implies the desired inequalities in the lemma with $u^{(\varepsilon)}$, $p^{(\varepsilon)}$,
and $U^{(\varepsilon)}$ in place of $u$, $p$,
and $U$.
To finish the proof, we let $\varepsilon \to 0$.
Indeed, we check that
$Du^{(\varepsilon)} \to Du$ and $p^{(\varepsilon)} \to p$ in $L_1(B_{r_1}^+)$.
Then, from the estimates for $(u^{(\varepsilon)},p^{(\varepsilon)})$, it follows that $\{Du^{(\varepsilon)}\}$ and $\{p^{(\varepsilon)}\}$ are Cauchy sequences in $L_\infty(B_r^+)$, and $\{D_{x'}u^{(\varepsilon)}\}$ and $\{U^{(\varepsilon)}\}$ are Cauchy sequences in $C^{1/2}(B_r^+)$.
By letting $\varepsilon \to 0$, we finally obtain the $L_\infty$-estimates for $Du$ and $p$, and the H\"{o}lder semi-norm estimates for $D_{x'}u$ and $U$.
\end{proof}

\section{$L_q$-estimates for Stokes system}
							\label{sec05}

Let $\cL$ be the elliptic operator from \eqref{eq0812_01} and $\Omega$ be a bounded Reifenberg flat domain satisfying Assumption \ref{assum1004}.
We consider
\begin{equation*}
\begin{cases}
\cL u + \nabla p = D_\alpha f_\alpha
\quad
&\text{in}\,\,\Omega,
\\
\operatorname{div} u = g
\quad
&\text{in}\,\,\Omega,
\\
u = 0
\quad
&\text{on} \,\, \partial \Omega,
\end{cases}
\end{equation*}
where $f_\alpha,g\in L_q(\Omega)$, $(g)_\Omega=0$, and $q \in (1,\infty)$.

\begin{proposition}
							\label{prop0719_1}
Let $\rho \in (0,1/120)$, $q \in (2,\infty)$, and $\Omega$ be a bounded domain with $\operatorname{diam}\Omega \le K$. Suppose that $(\cL, \Omega$) satisfies Assumptions \ref{assum1004} and \ref{assum0711_1} ($\rho$) with a constant $K_1$ in the estimate \eqref{eq0711_00}, and $(u,p) \in W_2^1(\Omega)^d \times L_2(\Omega)$ satisfies
\begin{equation}
							\label{eq10.49}
\begin{cases}
\cL u + \nabla p = D_\alpha f_\alpha
\quad
&\text{in}\,\,\Omega,
\\
\operatorname{div} u = g
\quad
&\text{in}\,\,\Omega,
\\
u = 0
\quad
&\text{on} \,\, \partial \Omega,
\end{cases}
\end{equation}
where $f_\alpha, g\in L_q(\Omega)$ and $(g)_\Omega=0$.
Then, we have the following.
\begin{enumerate}
\item[(i)] For $x_0 \in \Omega$ and $r \in (0, R_1/8]$ such that $B_r(x_0) \subset \Omega$,
there exist
$$
(w,p_1), (v,p_2) \in W_2^1(B_r(x_0))^d \times L_2(B_r(x_0))
$$
such that $(u,p) = (w,p_1) + (v,p_2)$ in $B_r(x_0)$ and
\begin{align}
							\label{eq0715_01}
&( |Dw|^2 + |p_1|^2 )_{B_r(x_0)}^{\frac{1}{2}} \nonumber\\
&\le N \rho^{\frac{1}{2\nu}} ( |Du|^2 + |p|^2 )^{\frac{1}{2}}_{\Omega_{8r}(x_0)}+N(|f_\alpha|^{2\mu} + |g|^{2\mu})_{\Omega_{8r}(x_0)}^{\frac{1}{2\mu}},\\
							\label{eq0715_02}
&\|Dv\|_{L_\infty(B_{r/2})} + \|p_2\|_{L_\infty(B_{r/2})}\nonumber
\\
&\le N (\rho^{\frac 1 {2\nu}}+1)( |Du|^2 + |p|^2 )^{\frac{1}{2}}_{\Omega_{8r}(x_0)} + N (|f_\alpha|^{2\mu} + |g|^{2\mu})_{\Omega_{8r}(x_0)}^{\frac{1}{2\mu}},
\end{align}
where $N>0$ and $\mu,\nu>1$ are constants depending only on $d$, $\delta$, $K_1$, and $q$, and $\mu$ satisfies $2\mu<q$ and $1/\mu + 1/\nu = 1$.

\item[(ii)] For $x_0 \in \partial \Omega$ and $r \in (0,R_1]$, there exist
$$
(w,p_1), (v,p_2) \in W_2^1(\Omega_{r/10}(x_0))^d \times L_2(\Omega_{r/10}(x_0))
$$
such that $(u,p) = (w,p_1) + (v,p_2)$ in $\Omega_{r/10}(x_0)$ and
\begin{align}
							\label{eq0718_01}
&( |Dw|^2 + |p_1|^2 )_{\Omega_{r/10}(x_0)}^{\frac{1}{2}} \nonumber\\
&\le N \rho^{\frac{1}{2\nu}} ( |Du|^2 + |p|^2 )^{\frac{1}{2}}_{\Omega_r(x_0)} + N (|f_\alpha|^{2\mu} + |g|^{2\mu})_{\Omega_r(x_0)}^{\frac{1}{2\mu}},\\
							\label{eq0718_02}
&\|Dv\|_{L_\infty(\Omega_{r/20}(x_0))} + \|p_2\|_{L_\infty(\Omega_{r/20}(x_0))}\nonumber
\\
&\le N (\rho^{\frac{1}{2\nu}}+1) ( |Du|^2 + |p|^2 )^{\frac{1}{2}}_{\Omega_r(x_0)} + N (|f_\alpha|^{2\mu} + |g|^{2\mu})_{\Omega_r(x_0)}^{\frac{1}{2\mu}},	
\end{align}
where $N>0$ and $\mu,\nu>1$ are constants depending only on $d$, $\delta$, $K_1$, and $q$, and $\mu$ satisfies $2\mu<q$ and $1/\mu + 1/\nu = 1$.
\end{enumerate}
\end{proposition}

\begin{proof}

{\em Case (i).}
Without loss of generality, we assume that $x_0 = 0$. By Assumption \ref{assum0711_1}, there is a coordinate system such that
\begin{equation}
							\label{eq14.07}
\dashint_{B_{r}} \Big| A^{\alpha\beta}(x_1,x') - \bar A^{\alpha\beta}(x_1)\Big| \, dx \le \rho,
\end{equation}
where
$$
\bar A^{\alpha\beta}(x_1)=\dashint_{B_{r}'} A^{\alpha\beta}(x_1,x') \, dx'.
$$
Let $\cL_0$ be the elliptic operator with the coefficients $\bar A^{\alpha\beta}$.

Recall that for $B_r$, Assumption \ref{assum0224_1} is satisfied with $K_1$ depending only on $d$. See Remark \ref{rem0229_1}. By Lemma \ref{lem0225_1}, there exists a unique solution $(w,p_1) \in W_2^1(B_r)^d \times L_2(B_r)$ with $(p_1)_{B_r} = 0$ satisfying
$$
\begin{cases}
\cL_0 w + \nabla p_1 = D_\alpha \big(f_\alpha
+(\bar A^{\alpha\beta}-A^{\alpha\beta})D_\beta u)
\quad
&\text{in}\,\,B_{r},
\\
\operatorname{div} w = g - (g)_{B_r}
\quad
&\text{in}\,\,B_r,
\\
w = 0\quad &\text{on}\,\,\partial B_r,
\end{cases}
$$
and
\begin{align*}
&\||Dw| + |p_1|\|_{L_2(B_r)} \le N \|f_\alpha+(\bar A-A)D u)\|_{L_2(B_r)} + N \|g-(g)_{B_r}\|_{L_2(B_r)}\\
&\le N \|f_\alpha\|_{L_2(B_r)}+ N \|\bar A-A\|_{L_{2\nu}(B_r)}\|D u\|_{L_{2\mu}(B_r)} + N \|g\|_{L_2(B_r)},
\end{align*}
where $N=N(\delta,d)$, $1/\mu + 1/\nu = 1$, $2\mu=\tilde q \in (2,\infty)$ is a number from Lemma \ref{lem0715_1} that depends only on $d$, $\delta$, $K_1$, and $q$. We used H\"older's inequality in the last inequality.
This together with \eqref{eq14.07}, the boundedness of $A^{ij}$ by $\delta^{-1}$, and Lemma \ref{lem0715_1} implies
\eqref{eq0715_01}.
Here, we use the fact that $r \le R_1/8 \le R_0/8$ and inequalities such as $(|D\bar{u}|^2)_{B_{8r}}^{1/2} \le (|Du|^2)_{\Omega_{8r}}^{1/2}$, where $\bar{u}$ is the zero extension of $u$ outside $\Omega$.

Now, set $(v,p_2) = (u,p) - (w,p_1)$, which satisfies
$$
\begin{cases}
\cL_0 v + \nabla p_2 = 0
\quad
&\text{in}\,\,B_r,
\\
\operatorname{div} v = (g)_{B_r}
\quad
&\text{in}\,\,B_r.
\end{cases}
$$
Then, by Lemma \ref{lem0829_2} (i) with $r=r/2$ and $R=r$ we have
$$
\||Dv|+|p_2|\|_{L_\infty(B_{r/2})} \le N (|Dv|^2)_{B_r}^{1/2} + N (|p_2|^2)_{B_r}^{1/2},
$$
where $N=N(\delta,d)$.
From this, the fact that $(v,p_2) = (u-w,p-p_1)$, and \eqref{eq0715_01}, we obtain \eqref{eq0715_02}.

{\em Case (ii).} Again, we assume that $x_0 = 0$.
Owing to Assumption \ref{assum0711_1}, we can find an orthogonal transformation to obtain
$$
 \{(x_1, x'):\rho r< x_1\}\cap B_r
 \subset\Omega\cap B_r
 \subset \{(x_1, x'):-\rho r<x_1\}\cap B_r
$$
as well as \eqref{eq14.07}.

Take a smooth function $\chi$ on $\bR$ such that
$$
\chi(x_1)\equiv 0\quad\text{for}\,\,x_1\le \rho r,
\quad \chi(x_1)\equiv 1\quad\text{for}\,\,x_1\ge  2\rho r,\quad
\text{and}
\quad
|\chi'| \le 2(\rho r)^{-1}.
$$
Let $\hat u=\chi u$, which vanishes on $B_r\cap \{x_1\le \rho r\}$. We see that $(\hat u,p)$ satisfies
\begin{equation}
                                    \label{eq17.23b}
\begin{cases}
\cL_0 \hat u+\nabla p = D_\alpha \big(f_\alpha+h_\alpha\big)
\quad
&\text{in}\,\,B_r\cap\{x_1>\rho r\},
\\
\operatorname{div}  \hat u = \chi g+\chi' u_1
\quad
&\text{in}\,\,B_r\cap\{x_1>\rho r\},
\\
\hat u = 0
\quad
&\text{on} \,\, B_r\cap \{x_1=\rho r\},
\end{cases}
\end{equation}
where
$$
h_\alpha= (\bar A^{\alpha\beta}-A^{\alpha\beta})D_\beta u + \bar A^{\alpha\beta}D_\beta((\chi-1) u).
$$
For $\tau \in [0,\infty)$, denote
$$
\widetilde{B}_r^+(\tau,0) = B_r(\tau,0) \cap \{x_1 > \tau\},
$$
where $0 \in \bR^{d-1}$.
Since $\rho\in (0,1/120)$, we have
\begin{equation}
							\label{eq0716_01}
\Omega_{r/20} \subset \Omega_{7r/120}(\rho r,0) \subset \Omega_{7r/60}(\rho r,0) \subset \Omega_{r/8} \subset \Omega_r
\end{equation}
and
\begin{equation}
							\label{eq0722_01}
\Omega_{r/10} \subset \Omega_{7r/60}(\rho r, 0),
\quad
\widetilde{B}^+_{7r/120}(\rho r,0)
\subset \widetilde{B}^+_{7r/60}(\rho r,0) \subset \Omega_{r/8} \cap \{x_1 > \rho r\}.
\end{equation}
By Lemma \ref{lem0225_1}, there is a unique solution
$$
(\hat w,\hat{p}_1)\in W^1_2(\widetilde{B}^+_{7r/60}(\rho r,0))^d \times L_2(\widetilde{B}^+_{7r/60}(\rho r,0))
$$
satisfying $(\hat{p}_1)_{\widetilde{B}^+_{7r/60}(\rho r,0)}=0$ and
\begin{equation}
							\label{eq0714_03}
\begin{cases}
\cL_0 \hat w+\nabla \hat{p}_1 = D_\alpha (f_\alpha+h_\alpha)
\quad
&\text{in}\,\,\widetilde{B}^+_{7r/60}(\rho r,0),
\\
\operatorname{div} \hat w = \chi g+\chi' u_1-\big(\chi g+\chi' u_1\big)_{\widetilde{B}^+_{7r/60}(\rho r,0)}
\quad
&\text{in}\,\,\widetilde{B}^+_{7r/60}(\rho r,0),
\\
\hat w = 0
\quad
&\text{on} \,\, \partial \widetilde{B}^+_{7r/60}(\rho r,0).
\end{cases}
\end{equation}
Moreover, it holds that
\begin{align}
                                \label{eq3.01}
&\|D\hat w\|_{L_2} + \|\hat{p}_1\|_{L_2} \le N\big(\|f_\alpha\|_{L_2}
+\|h_\alpha\|_{L_2}
+\|\chi g+\chi' u\|_{L_2}\big)\nonumber\\
&\le N\big(\|f_\alpha\|_{L_2}
+\|(\bar A-A)Du\|_{L_2}+\|D((\chi-1)u)\|_{L_2}
+\|g\|_{L_2}+\|\chi' u\|_{L_2}\big),
\end{align}
where $\|\cdot\|_{L_2}=\|\cdot\|_{L_2(\widetilde{B}^+_{7r/60}(\rho r,0))}$ and $N = N(d,\delta)$.
Since $\chi-1$ is supported on $\{x_1\le 2\rho r\}$, H\"older's inequality and \eqref{eq14.07} imply that
\begin{equation}
							\label{eq0328_06}
\|(\chi-1)Du\|_{L_2(\widetilde{B}^+_{7r/60}(\rho r,0))} \le N\rho^{\frac 1 {2\nu}}r^{\frac d {2\nu}}
\|Du\|_{L_{2\mu}(\Omega_{r/8})},
\end{equation}
and
\begin{equation}
							\label{eq14.28}
\|(\bar A-A)Du\|_{L_2(\widetilde{B}^+_{7r/60}(\rho r,0))} \le N\rho^{\frac 1 {2\nu}}r^{\frac d {2\nu}}
\|Du\|_{L_{2\mu}(\Omega_{r/8})},
\end{equation}
where $1/\mu + 1/\nu = 1$ and $2\mu=\tilde q \in (2,\infty)$ is a number from Lemma \ref{lem0715_1} depending only on $d$, $K_1$, and $q$.
Using H\"older's inequality again with the same $\mu$ and $\nu$, along with the fact that $\chi'$ is supported on $\{\rho r \le x_1 \le 2\rho r\}$, we have
\begin{align}
							\label{eq0328_07}
\|\chi' u\|_{L_2(\widetilde{B}^+_{7r/60}(\rho r,0))}
&\le N \rho^{\frac{1}{2\nu}} r^{\frac{d}{2\nu}}
\|\chi' u\|_{L_{2\mu}(\widetilde{B}^+_{7r/60}(\rho r,0))}\nonumber\\
&\le N \rho^{\frac{1}{2\nu}} r^{\frac{d}{2\nu}} \|Du\|_{L_{2\mu}(\Omega_{r/8})},
\end{align}
where the last inequality follows from Hardy's inequality, using the boundary condition $u = 0$ on $\partial\Omega$ and the observation that
$$
|\chi'| \le N(x_1 - \phi(x'))^{-1}
$$
for $(x_1,x') \in \Omega_r$. Here, $\phi(x')$ is the largest number such that $(\phi(x'),x') \in \partial \Omega$.
Inserting \eqref{eq0328_06}, \eqref{eq14.28}, and \eqref{eq0328_07} into \eqref{eq3.01} gives
\begin{multline}
							\label{eq21.52h}
\|D\hat{w}\|_{L_2(\tilde{B}^+_{7r/60}(\rho r,0))} + \|\hat{p}_1\|_{L_2(\tilde{B}^+_{7r/60}(\rho r,0))}
\\
\le N \rho^{\frac{1}{2\nu}} r^{\frac{d}{2\nu}} \|Du\|_{L_{2\mu}(\Omega_{r/8})}
+ N \|f_\alpha\|_{L_2(\Omega_{r/8})} + N \|g\|_{L_2(\Omega_{r/8})},
\end{multline}
where $N=N(d,\delta)$.
We extend $\hat w$ to be zero  in $\Omega_{7r/60}(\rho r,0)\cap \{x_1<\rho r\}$ so that $\hat w\in W^1_2(\Omega_{7r/60}(\rho r,0))$, and we set
$$
w=\hat w+(1-\chi)u.
$$
We also set
$$
p_1 =
\begin{cases}
\hat{p}_1 \quad &\text{in} \,\, \tilde{B}^+_{7r/60}(\rho r,0),
\\
p \quad &\text{in} \,\,\Omega_{7r/60}(\rho r,0)\cap \{x_1<\rho r\}.
\end{cases}
$$
By the same reasoning as in \eqref{eq0328_06} and \eqref{eq0328_07}, using \eqref{eq0716_01} we have
$$
\|D\left((1-\chi) u\right)\|_{L_2(\Omega_{7r/60}(\rho r,0))} \le N \rho^{\frac{1}{2\nu}} r^{\frac{d}{2 \nu}} \|Du\|_{L_{2\mu}(\Omega_{r/8})}
$$
and
$$
\|p_1\|_{L_2(\Omega_{7r/60}(\rho r,0))} \le \|\hat{p}_1\|_{L_2(\tilde{B}^+_{7r/60}(\rho r,0))} + N\rho^{\frac{1}{2\nu}} r^{\frac{d}{2\mu}} \|p\|_{L_{2\mu}(\Omega_{r/8})},
$$
where $N=N(d)$.
From these inequalities and \eqref{eq21.52h}, we deduce that
$$
\|Dw\|_{L_2(\Omega_{7r/60}(\rho r,0))} + \|p_1\|_{L_2(\Omega_{7r/60}(\rho r,0))} \le N \|f_\alpha\|_{L_2(\Omega_{r/8})} + N \|g\|_{L_2(\Omega_{r/8})}
$$
$$
+ N \rho^{\frac 1 {2\nu}} r^{\frac{d}{2\nu}} \left(\|Du\|_{L_{2\mu}(\Omega_{r/8})} + \|p\|_{L_{2\mu}(\Omega_{r/8})}\right),
$$
where $N=N(d,\delta)$.
This combined with Lemma \ref{lem0715_1} and the inequality $7r/60 \le 7R_1/60 \le R_0/8$ shows that
\begin{equation}
            \label{eq18.34h}
(|Dw|^2 + |p_1|^2)_{\Omega_{7r/60}(\rho r,0)}^{\frac 1 2}
\le N \rho^{\frac 1 {2\nu}}(|Du|^2 + |p|^2)_{\Omega_r}^{\frac 1  2}+
N(|f_\alpha|^{2\mu}+|g|^{2\mu})_{\Omega_r}^{\frac 1 {2\mu}},
\end{equation}
where $(N,\mu)=(N,\mu)(d,\delta, K_1, q)$, $1/\mu+1/\nu = 1$.
Using this inequality, \eqref{eq0722_01}, and the fact that $|\Omega_{r/10}| \ge N(d) |\Omega_{7r/60}(\rho r,0)|$ if $\rho < 1/120$, we obtain \eqref{eq0718_01}.

Next, we define $v= u- w$ $(=\chi u-\hat w)$ and $p_2=p-p_1$ in $\Omega_{7r/60}(\rho r,0)$.
In particular,
$$
p_2 = 0
\quad
\text{in} \,\,\Omega_{7r/60}(\rho r, 0) \cap \{x_1 < \rho r\}.
$$
From \eqref{eq17.23b} and \eqref{eq0714_03}, we see that $(v,p_2)$ satisfies
\begin{equation*}
\begin{cases}
\cL_0 v+\nabla p_2 = 0
\quad
&\text{in}\,\,\widetilde{B}^+_{7r/60}(\rho r,0),
\\
\operatorname{div} v = \big(\chi g+\chi' u_1\big)_{\widetilde{B}^+_{7r/60}(\rho r,0)}
\quad
&\text{in}\,\,\widetilde{B}^+_{7r/60}(\rho r,0),
\\
v = 0
\quad
&\text{on} \,\, B_{7r/60}(\rho r,0)\cap \{x_1=\rho r\}.
\end{cases}
\end{equation*}
Denote
$$
\cD_1=\Omega_{r/20}\cap \{x_1\le \rho r\},
\,\,
\cD_2=\Omega_{r/20} \cap \{x_1> \rho r\},
\,\,\text{and}\,\,
\cD_3=\widetilde{B}^+_{7r/120}(\rho r,0).
$$
Note that, by \eqref{eq0716_01}, $\cD_2\subset \cD_3$, and $(v,p_2)=(0,0)$ in $\cD_1$.
By applying Lemma \ref{lem0829_2} (ii) with $r= 7r/120$ and $R=7r/60$,
we get
\begin{align*}
&\||Dv| + |p_2|\|_{L_\infty(\Omega_{r/20})} = \||Dv| + |p_2|\|_{L_\infty(\cD_2)}
\le \||Dv| + |p_2|\|_{L_\infty(\cD_3)}\\
&\le N r^{-d/2}\|Dv\|_{L_2(\widetilde{B}^+_{7r/60}(\rho r,0))} + N r^{-d/2} \|p_2\|_{L_2(\widetilde{B}^+_{7r/60}(\rho r,0))}.
\end{align*}
This together with the equality $(u,p) = (w,p_1) + (v,p_2)$ in $\Omega_{7r/60}(\rho r,0)$ and \eqref{eq18.34h} gives
\begin{align*}
\||Dv| + |p_2|\|_{L_\infty(\Omega_{r/20})} &\le N \left(|Du|^2+|p|^2 \right)_{\Omega_{7r/60}(\rho r, 0)}^{\frac 1 2} + N \left(|Dw|^2+|p_1|^2\right)_{\Omega_{7r/60}(\rho r, 0)}^{\frac 1 2}
\\
&\le N(\rho^{\frac 1 {2\nu}}+1) (|Du|^2 + |p|^2)_{\Omega_r}^{\frac 1  2}+
N(|f_\alpha|^{2\mu}+|g|^{2\mu})_{\Omega_r}^{\frac 1 {2\mu}}.
\end{align*}
Hence, the inequality \eqref{eq0718_02} is proved.
\end{proof}

\begin{corollary}
							\label{cor0719_1}
Let $\rho \in (0,1/320)$, $q \in (2,\infty)$, and $\Omega$ be a bounded domain with $\operatorname{diam} \Omega \le K$. Suppose that $(\cL, \Omega)$ satisfies Assumptions \ref{assum1004} and \ref{assum0711_1} ($\rho$) with a constant $K_1$ in the estimate \eqref{eq0711_00}, and $(u,p) \in W_2^1(\Omega)^d \times L_2(\Omega)$ satisfies
\begin{equation*}
\begin{cases}
\cL u + \nabla p = D_\alpha f_\alpha
\quad
&\text{in}\,\,\Omega,
\\
\operatorname{div} u = g
\quad
&\text{in}\,\,\Omega,
\\
u = 0
\quad
&\text{on} \,\, \partial \Omega,
\end{cases}
\end{equation*}
where $f_\alpha, g\in L_q(\Omega)$ and $(g)_\Omega=0$.
Then, for any $x_0 \in \bar{\Omega}$ and $r \in (0,R_1]$, there exist
$$
(w,p_1), (v,p_2) \in W_2^1(\Omega_{r/40}(x_0))^d \times L_2(\Omega_{r/40}(x_0))
$$
such that $(u,p) = (w,p_1) + (v,p_2)$ in $\Omega_{r/40}(x_0)$ and
$$
( |Dw|^2 + |p_1|^2 )_{\Omega_{r/40}(x_0)}^{\frac{1}{2}} \le N \rho^{\frac{1}{2\nu}} ( |Du|^2 + |p|^2 )^{\frac{1}{2}}_{\Omega_r(x_0)} + N (|f_\alpha|^{2\mu} + |g|^{2\mu})_{\Omega_r(x_0)}^{\frac{1}{2\mu}},
$$
\begin{multline*}
\|Dv\|_{L_\infty(\Omega_{r/80}(x_0))} + \|p_2\|_{L_\infty(\Omega_{r/80}(x_0))}
\\
\le N (\rho^{\frac{1}{2\nu}}+1) ( |Du|^2 + |p|^2 )^{\frac{1}{2}}_{\Omega_r(x_0)} + N (|f_\alpha|^{2\mu} + |g|^{2\mu})_{\Omega_r(x_0)}^{\frac{1}{2\mu}},	
\end{multline*}
where $N>0$ and $\mu,\nu>1$ are constants depending only on $d$, $\delta$, $K_1$, and $q$, and $\mu$ satisfies $2\mu<q$ and $1/\mu + 1/\nu = 1$.
\end{corollary}

\begin{proof}
Without loss of generality, we may assume that $x_0 = 0$.
We consider the following two cases.
$$
B_{r/40} \subset \Omega,
\quad
B_{r/40} \cap \partial \Omega \neq \emptyset.
$$
If $B_{r/40} \subset \Omega$, then by Proposition \ref{prop0719_1} (i) with the fact that $r/40 \le R_1/8$, there exist
$$
(w,p_1), (v,p_2) \in W_2^1(B_{r/40})^d \times L_2(B_{r/40})
$$
such that $(u,p) = (w,p_1) + (v,p_2)$ in $B_{r/40}(x_0)$ and
$$
( |Dw|^2 + |p_1|^2 )_{B_{r/40}(x_0)}^{\frac{1}{2}}
\le N \rho^{\frac{1}{2\nu}} ( |Du|^2 + |p|^2 )^{\frac{1}{2}}_{\Omega_{r/5}(x_0)}+N(|f_\alpha|^{2\mu} + |g|^{2\mu})_{\Omega_{r/5}(x_0)}^{\frac{1}{2\mu}},
$$
\begin{multline*}
\|Dv\|_{L_\infty(B_{r/80})} + \|p_2\|_{L_\infty(B_{r/80})}
\\
\le N (\rho^{\frac 1 {2\nu}}+1)( |Du|^2 + |p|^2 )^{\frac{1}{2}}_{\Omega_{r/5}(x_0)} + N (|f_\alpha|^{2\mu} + |g|^{2\mu})_{\Omega_{r/5}(x_0)}^{\frac{1}{2\mu}},
\end{multline*}
where $N=N(d,\delta)$.
These and an inequality similar to those in \eqref{eq0719_04} below show the desired inequalities above.

If $B_{r/40} \cap \partial \Omega \neq \emptyset$,
we find $y_0 \in \partial \Omega$ such that $\operatorname{dist}(0,\partial\Omega) = |y_0|$. Then
\begin{equation}
							\label{eq0719_03}
\Omega_{r/40}\subset \Omega_{3r/40}(y_0),\quad
\Omega_{r/80} \subset \Omega_{3r/80}(y_0) \subset \Omega_{3r/4}(y_0) \subset \Omega_r,	
\end{equation}
and
\begin{equation}
							\label{eq0719_04}
\frac{|\Omega_{3r/40}(y_0)|}{|\Omega_{r/40}|} \le N(d),
\quad
\frac{|\Omega_r|}{|\Omega_{3r/4}(y_0)|} \le N(d).	
\end{equation}
From Proposition \ref{prop0719_1} (ii) with $3r/4$ in place of $r$ it follows that
there exist
$$
(w,p_1), (v,p_2) \in W_2^1(\Omega_{3r/40}(y_0))^d \times L_2(\Omega_{3r/40}(y_0))
$$
such that $(u,p) = (w,p_1) + (v,p_2)$ in $\Omega_{3r/40}(y_0)$ and
\begin{multline*}
( |Dw|^2 + |p_1|^2 )_{\Omega_{3r/40}(y_0)}^{\frac{1}{2}} \\
\le N \rho^{\frac{1}{2\nu}} ( |Du|^2 + |p|^2 )^{\frac{1}{2}}_{\Omega_{3r/4}(y_0)} + N (|f_\alpha|^{2\mu} + |g|^{2\mu})_{\Omega_{3r/4}(y_0)}^{\frac{1}{2\mu}},
\end{multline*}
\begin{multline*}
\|Dv\|_{L_\infty(\Omega_{3r/80}(y_0))} + \|p_2\|_{L_\infty(\Omega_{3r/80}(y_0))}
\\
\le N (\rho^{\frac{1}{2\nu}}+1) ( |Du|^2 + |p|^2 )^{\frac{1}{2}}_{\Omega_{3r/4}(y_0)} + N (|f_\alpha|^{2\mu} + |g|^{2\mu})_{\Omega_{3r/4}(y_0)}^{\frac{1}{2\mu}},	
\end{multline*}
where $(N,\mu)=(N,\mu)(d,\delta,K_1,q)$, $1/\mu + 1/\nu = 1$.
From the above inequalities, \eqref{eq0719_03}, and \eqref{eq0719_04}, we obtain the desired conclusion.
\end{proof}

The maximal function of $f$ in $\bR^d$ is defined by
\begin{equation*}
\cM f(x) = \sup_{x_0\in \bR^d, B_r(x_0) \ni x} \dashint_{B_r(x_0)} |f(y)| \, dy.
\end{equation*}
Throughout this paper, by $\cM (f)$ we mean $\cM (fI_\Omega)$ if $f$ is defined on $\Omega \subset \bR^d$.
For $q\in (2,\infty)$, denote
\begin{align*}
\cA(s)& = \{ x \in \Omega: \left(\cM(|Du|^2 + |p|^2)(x)\right)^{1/2} > s\},\\
\cB(s)& = \{ x \in \Omega: \rho^{-1/(2\nu)}\left( \cM( |f_\alpha|^{2\mu} + |g|^{2\mu} )(x)\right)^{1/(2\mu)}\\
 &\quad \quad+ \left( \cM( |Du|^2 + |p|^2 )(x) \right)^{1/2} > s\},
\end{align*}
where $\mu,\nu \in (1,\infty)$ are from Corollary \ref{cor0719_1} satisfying $2\mu<q$ and $1/\mu + 1/\nu = 1$.

\begin{lemma}
                                \label{lem4.5}
Let $\rho \in (0,1/960)$, $q \in (2,\infty)$, and $\Omega$ be a bounded domain with $\operatorname{diam} \Omega \le K$. Suppose that $(\cL, \Omega)$ satisfies Assumptions \ref{assum1004} and \ref{assum0711_1} ($\rho$) with a constant $K_1$ in the estimate \eqref{eq0711_00}, and $(u,p) \in W_2^1(\Omega)^d \times L_2(\Omega)$ satisfies
\begin{equation*}
\begin{cases}
\cL u + \nabla p = D_\alpha f_\alpha
\quad
&\text{in}\,\,\Omega,
\\
\operatorname{div} u = g
\quad
&\text{in}\,\,\Omega,
\\
u = 0
\quad
&\text{on} \,\, \partial \Omega,
\end{cases}
\end{equation*}
where $f_\alpha, g\in L_q(\Omega)$ and $(g)_\Omega=0$.
Then, there exists a constant $\kappa = \kappa(d,\delta,K_1,q) > 1$ such that the following holds:
For $x_0 \in \bar{\Omega}$, if
\begin{equation}
							\label{eq0720_08}
|\Omega_{R/240}(x_0) \cap \cA(\kappa s)| > \rho^{\frac{1}{\nu}}|\Omega_{R/240}(x_0)|,
\quad
R \in (0,R_1],
\quad
s > 0,
\end{equation}
then we have
$$
\Omega_{R/240}(x_0) \subset \cB(s).
$$
\end{lemma}

\begin{proof}
By scaling and translating the coordinates, we may assume that $s = 1$ and $x_0 = 0$.
We prove by contradiction.
Suppose that there exists $z_0 \in \Omega_{R/240}$ such that $z_0 \notin \cB(1)$. That is, $z_0 \in \Omega_{R/240}$ and
\begin{equation}
							\label{eq0719_01}
\rho^{-1/(2\nu)}\left( \cM( |f_\alpha|^{2\mu} + |g|^{2\mu} )(z_0)\right)^{1/(2\mu)} + \left( \cM( |Du|^2 + |p|^2 )(z_0) \right)^{1/2} \le 1.
\end{equation}
This implies that, for all $r > 0$,
\begin{equation}
							\label{eq0720_01}
\left( |f_\alpha|^{2\mu} I_\Omega + |g|^{2\mu}I_\Omega  \right)_{B_r(z_0)}^{1/\mu} \le \rho^{1/\nu}
\end{equation}
and
\begin{equation}
							\label{eq0720_02}
\left( |Du|^2 I_\Omega + |p|^2 I_\Omega  \right)_{B_r(z_0)} \le 1.
\end{equation}
By Corollary \ref{cor0719_1}, there exist
$$
(w,p_1), (v,p_2) \in W_2^1(\Omega_{R/40}(z_0))^d \times L_2(\Omega_{R/40}(z_0))
$$
such that $(u,p) = (w,p_1) + (v,p_2)$ in $\Omega_{R/40}(z_0)$ and
\begin{multline}
							\label{eq0720_03}
( |Dw|^2 + |p_1|^2 )_{\Omega_{R/40}(z_0)}^{\frac{1}{2}} \\
\le N \rho^{\frac{1}{2\nu}} ( |Du|^2 + |p|^2 )^{\frac{1}{2}}_{\Omega_R(z_0)} + N (|f_\alpha|^{2\mu} + |g|^{2\mu})_{\Omega_R(z_0)}^{\frac{1}{2\mu}},
\end{multline}
\begin{multline}
							\label{eq0720_04}
\|Dv\|_{L_\infty(\Omega_{R/80}(z_0))} + \|p_2\|_{L_\infty(\Omega_{R/80}(z_0))}
\\
\le N (\rho^{\frac{1}{2\nu}}+1) ( |Du|^2 + |p|^2 )^{\frac{1}{2}}_{\Omega_R(z_0)} + N (|f_\alpha|^{2\mu} + |g|^{2\mu})_{\Omega_R(z_0)}^{\frac{1}{2\mu}},	
\end{multline}
where $N=N(d,\delta,K_1,q)$. It follows from \eqref{eq0720_03}, \eqref{eq0720_01}, and \eqref{eq0720_02} that
\begin{equation}
							\label{eq2.19}
( |Dw|^2 + |p_1|^2 )_{\Omega_{R/40}(z_0)}^{\frac{1}{2}} \le N\rho^{\frac 1 {2\nu}}.
\end{equation}

Next we show that for sufficiently large $\kappa>0$, that depends only on $d$, $K_1$, and $q$, we have
\begin{multline}
							\label{eq7020_07}
\{\cM (|Du|^2 + |p|^2) > \kappa^2 \} \cap \Omega_{R/240}
\\
\subset \left\{\cM \left( |Dw|^2 I_{\Omega_{R/80}(z_0)} + |p_1|^2 I_{\Omega_{R/80}(z_0)} \right) > \kappa^2/3 \right\} \cap \Omega_{R/240}.
\end{multline}
To prove this, let $y$ be a point belonging to
\begin{equation}
							\label{eq0719_02}
\left\{\cM \left(|Dw|^2   I_{\Omega_{R/80}(z_0)} + |p_1|^2 I_{\Omega_{R/80}(z_0)}\right) \le \kappa^2/3 \right\} \cap \Omega_{R/240}.
\end{equation}
By the triangle inequality, we have $B_r(y) \cap \Omega \subset \Omega_{r+R/120}(z_0)$. For $r \ge R/240$, by \eqref{eq0720_02},
\begin{align}
							\label{eq0720_06}
&\dashint_{B_r(y)} (|Du|^2 + |p|^2)I_\Omega \le \frac{1}{|B_r(y)|} \int_{B_{r+R/120}(z_0)} (|Du|^2 + |p|^2)I_{\Omega}\nonumber
\\
&\le \frac{|B_{r+R/120}(z_0)|}{|B_r(y)|} = \left(\frac{r+R/120}{r}\right)^d \le 3^d.
\end{align}
For $r < R/240$, by the triangle inequality, we have
$\Omega_r(y) \subset \Omega_{R/80}(z_0)$.
Since \eqref{eq0719_02} means that
$$
\dashint_{B_r(y)} |Dw|^2 I_{\Omega_{R/80}(z_0)}   + |p_1|^2 I_{\Omega_{R/80}(z_0)} \, dz \le \kappa^2/3
$$
for any $r > 0$, we have
\begin{align}
							\label{eq0720_05}
&\dashint_{B_r(y)} \left(|Du|^2 + |p|^2\right)I_\Omega
=  \dashint_{B_r(y)} \left(|Du|^2 + |p|^2\right)I_{\Omega_{r}(y)}\nonumber
\\
&\le 2 \dashint_{B_r(y)} |Dw|^2 I_{\Omega_{R/80}(z_0)} + |p_1|^2 I_{\Omega_{R/80}(z_0)} + (|Dv|^2+|p_2|^2)I_{\Omega_{R/80}(z_0)}\nonumber
\\
&\le 2 \kappa^2/3 + 2\dashint_{B_r(y)} (|Dv|^2+|p_2|^2)I_{\Omega_{R/80}(z_0)}.
\end{align}
By \eqref{eq0720_04}, on $\Omega_r(y) \subset \Omega_{R/80}(z_0)$, we have
$$
|Dv|^2 + |p_2|^2 \le N(\rho^{\frac{1}{\nu}}+1) ( |Du|^2 + |p|^2 )_{\Omega_R(z_0)} + N (|f_\alpha|^{2\mu} + |g|^{2\mu})_{\Omega_R(z_0)}^{\frac{1}{\mu}}.
$$
From this, \eqref{eq0720_05}, and \eqref{eq0719_01}, it follows that
\begin{equation}
							\label{eq0720_09}
\frac{1}{|B_r(y)|}\int_{B_r(x)} \left(|Du|^2 + |p|^2\right)I_{\Omega}
\le 2 \kappa^2/3+ N_1,
\end{equation}
where $N_1=N_1(d, \delta,K_1, q)$.
From \eqref{eq0720_09} and \eqref{eq0720_06}, we arrive at \eqref{eq7020_07} provided that $\kappa^2 \ge \max\{3^d, 3N_1\}$.

Now by \eqref{eq7020_07}, \eqref{eq2.19}, and the Hardy-Littlewood inequality, we see that
\begin{align*}
&|\{ \cM ( |Du|^2+|p|^2 ) > \kappa^2 \} \cap \Omega_{R/240}|
\le \frac{N}{\kappa^2/3} \int_{\bR^d} |Dw|^2 I_{\Omega_{R/80}(z_0)} + |p_1|^2 I_{\Omega_{R/80}(z_0)}\\
&\le \frac{N \rho^{1/\nu}|\Omega_{R/40}(z_0)|}{\kappa^2/3} \le \rho^{1/\nu} |\Omega_{R/240}|,
\end{align*}
provided that $\kappa$ is sufficiently large depending only on $d$, $\delta$, $K_1$, and $q$.
This contracts the assumption
\eqref{eq0720_08}.
The lemma is proved.
\end{proof}

We are now ready to complete the proof of Theorem \ref{thm1}.
\begin{proof}[Proof of Theorem \ref{thm1}]
We may assume that $f \equiv 0$.
Indeed, for $B_R \supseteq \Omega$, we find $w \in W_{q_1}^2(B_R)$ such that $\Delta w = f 1_{\Omega}$ in $B_R$ and $w|_{\partial B_R} = 0$.
Then, we consider
$$
\cL u + \nabla p = D_\alpha\left( D_\alpha w + f_\alpha\right),
$$
for which from the Sobolev embedding theorem and the well-known $L_{q_1}$-estimate for the Laplace equation we have
$$
\|D_\alpha w\|_{L_q(\Omega)} \le \|D_\alpha w\|_{L_q(B_R)} \le N \|w\|_{W_{q_1}^2(B_R)}
\le N\|f\|_{L_{q_1}(\Omega)}.
$$
Also, owing to Lemma \ref{lem0225_1}, we only need to consider the case when $q\neq 2$.

{\em Case 1: $q>2$}.
Note that $\Omega$ is bounded and, owing to Assumption \ref{assum0711_1} (also see Remark \ref{rem0229_2}), satisfies Assumption \ref{assum0224_1} with $K_1 = K_1(d, R_0, K)$.
Thus by Lemma \ref{lem0225_1}, there exists a unique solution
$(u,p) \in  W_2^1(\Omega)^d \times L_2(\Omega)$.
We prove that this $(u,p)$ is indeed in $W_q^1(\Omega)^d \times L_q(\Omega)$ and satisfies \eqref{eq0328_02}.

Let $\kappa = \kappa(d,\delta,K_1,q)$ be the constant in Lemma \ref{lem4.5}.
Then, for any $s>0$, by the Hardy-Littlewood inequality,
\begin{equation}
                \label{eq12.16}
|\cA(\kappa s)|\le N_0(\kappa s)^{-2}\||Du|^2+p^2\|_{L_1(\Omega)}
\le N_0(\kappa s)^{-2} \big(
\|Du\|^2_{L_2(\Omega)}+\|p\|^2_{L_2(\Omega)}\big),
\end{equation}
where $N_0=N_0(d)$.
From \eqref{eq12.16}, Lemma \ref{lem4.5}, and a result from measure theory on the ``crawling of ink spots,'' which
can be found in \cite{MR579490} or \cite[Section 2]{MR563790}, we have the following upper bound of $\cA$. For any
$$
\rho\in (0,1/960)\quad\text{and}\quad \kappa s\ge N_0^{1/2}\rho^{-1/(2\nu)} |B_{R_1/240}|^{-1/2}\big(\|Du\|^2_{L_2(\Omega)}+\|p\|^2_{L_2(\Omega)}\big)^{1/2},
$$
we have
\begin{equation}
                                \label{eq12.24}
|\cA(\kappa s)|\le N(d)\rho^{1/\nu}|\cB(s)|.
\end{equation}
Recall the elementary identity
\begin{equation}
							\label{eq0812_02}
\|f\|_{L_q(\Omega)}^q=q\int_0^\infty|\{x\in \Omega\,:\,|f(x)|>s\}|s^{q-1}\,ds.
\end{equation}
By using \eqref{eq12.16} when
$$
\kappa s\in \left(0,N_0^{1/2}\rho^{-1/(2\nu)}|B_{R_1/240}                                                                                                                                                                                                                                                  |^{-1/2}\big(\|Du\|^2_{L_2(\Omega)}
+\|p\|^2_{L_2(\Omega)}\big)^{1/2}\right)
$$
and \eqref{eq12.24} otherwise,
we get for any sufficiently large $S>0$,
\begin{align}
                    \label{eq1.50}
&\int_0^{\kappa S}|\cA(s)|s^{q-1}\,ds
=\int_0^{S}|\cA(\kappa s)|\kappa^q s^{q-1}\,ds\nonumber\\
&\le N\rho^{(2-q)/(2\nu)}
\big(\|Du\|^q_{L_2(\Omega)}+\|p\|^q_{L_2(\Omega)}\big)+N\rho^{1/\nu}\int_0^S \cB(s)s^{q-1}\,ds\nonumber\\
&\le N_1\rho^{(2-q)/(2\nu)}\big(\|Du\|^q_{L_2(\Omega)}+\|p\|^q_{L_2(\Omega)}
+\|f_\alpha\|_{L_q(\Omega)}^q+\|g\|_{L_q(\Omega)}^q\big)\nonumber\\
&\quad+N_2\rho^{1/\nu}
\int_0^{S}|\cA(s)|s^{q-1}\,ds,
\end{align}
where we used the Hardy-Littlewood maximal function theorem in the last inequality (recall $q>2\mu$), and $N_2$ depends only on $d$, $\delta$,  $K_1$, and $q$, and $N_1$ depends also on $R_1$.
Note that since $\Omega$ is bounded, we have
$$
\int_0^{\kappa S} |\cA (s)| s^{q-1} \, ds < \infty.
$$
Then by replacing $S$ by $\kappa S$ in the last integral in \eqref{eq1.50} and taking $\rho$ sufficiently small depending on $d$, $\delta$,  $K_1$, and $q$, we obtain from \eqref{eq1.50} that
$$
\int_0^{\kappa S}|\cA(s)|s^{q-1}\,ds\le
N\big(\|Du\|^q_{L_2(\Omega)}+\|p\|^q_{L_2(\Omega)}
+\|f_\alpha\|_{L_q(\Omega)}^q+\|g\|_{L_q(\Omega)}^q\big),
$$
where $N$ depends only on $d$, $\delta$, $K_1$, $q$, and $R_1$.
Now, let $S\to \infty$, and use the identity \eqref{eq0812_02} and Lemma \ref{lem0225_1} to obtain
$$
\| \cM ( |Du|^2 + |p|^2 ) \|_{L_{q/2}(\Omega)}^{q/2} \le N \|f_\alpha\|_{L_q(\Omega)}^q + N\|g\|_{L_q(\Omega)}^q,
$$
where $N=N(d,\delta, R_1, K_1,K, q)$.
From this, we finally see that $(u,p)$ is in $ W_q^1(\Omega)^d \times L_q(\Omega)$ and satisfies \eqref{eq0328_02} because by Lebesgue differentiation theorem
$$
|Du|^2 + |p|^2 \le \cM( |Du|^2 + |p|^2 )
$$
for a.e. $x \in \Omega$.

{\em Case 2: $q\in (1,2)$.} First we prove \eqref{eq0328_02} by using a duality argument. Let $q'=q/(q-1)\in (2,\infty)$ and $\rho = \rho(d,\delta, K_1,q')$ from Case 1. Then, for any $\eta=(\eta_{\alpha})$, where $\eta_\alpha\in L_{q'}(\Omega)^d$ for $\alpha=1,\ldots,d$, and any $h\in L_{q'}(\Omega)$, there exists a unique solution $(v,\pi) \in W_{q'}^{1}(\Omega)^d \times L_{q'}(\Omega)$ with $(\pi)_\Omega = 0$ satisfying
\begin{equation}
                                \label{eq9.44}
\begin{cases}
D_\beta(A^{\alpha\beta}_{\operatorname{tr}} D_\alpha v) + \nabla \pi = D_\alpha \eta_\alpha
\quad
&\text{in}\,\,\Omega,
\\
\operatorname{div} v = h-(h)_\Omega
\quad
&\text{in}\,\,\Omega,
\\
v= 0\quad
&\text{on}\,\,\partial\Omega,
\end{cases}
\end{equation}
where $A^{\alpha\beta}_{\operatorname{tr}}$ is the transpose of the matrix $A^{\alpha\beta}$ for each $\alpha,\beta=1,\ldots,d$. We also have 
\begin{equation}
                                \label{eq10.55}
\|Dv\|_{L_{q'}(\Omega)} + \|\pi\|_{L_{q'}(\Omega)} \le N \big(\|\eta_\alpha\|_{L_{q'}(\Omega)}+\|h\|_{L_{q'}(\Omega)}\big),
\end{equation}
where $N = N(d,\delta,R_1, K_1,  K,q')$. Now we test \eqref{eq9.44} by $u$ to obtain
$$
\int_\Omega  \eta_\alpha  \cdot D_\alpha u\,dx=\int_\Omega \big(D_\beta u \cdot A^{\alpha\beta}_{\operatorname{tr}}  D_\alpha v +\pi\, \operatorname{div} u\big)\,dx=
\int_\Omega \big(f_\alpha \cdot D_\alpha v-ph+\pi g\big)\,dx.
$$
From this and \eqref{eq10.55}, we get 
$$
\left|\int_\Omega \big(\eta_\alpha \cdot D_\alpha u+ph\big)\,dx\right|\le N\big(\|\eta_\alpha\|_{L_{q'}(\Omega)}+\|h\|_{L_{q'}(\Omega)}\big)
\big(\|f_\alpha\|_{L_q(\Omega)}
+\|g\|_{L_q(\Omega)}\big).
$$
Since $\eta\in \big(L_{q'}(\Omega)\big)^{d^2}$ and $h\in L_{q'}(\Omega)$ are arbitrary, we obtain \eqref{eq0328_02}.

For the solvability, for $k>0$, we consider the equation with
\begin{equation}
                    \label{eq1.46}
f^k_\alpha:=\max\{-k,\min\{f_\alpha,k\}\},\quad g^k:=\max\{-k,\min\{g,k\}\}
\end{equation}
in place of $f_\alpha$ and $g$. Since $f_\alpha^k,g^k\in L_2(\Omega)$, by Lemma \ref{lem0225_1}, there is a unique solution $(u_k,p_k)\in \mathring W^1_2(\Omega)^d\times L_2(\Omega)$ satisfying $(p_k)_\Omega=0$
and
\begin{equation}
                            \label{eq1.48}
\begin{cases}
\cL u_k + \nabla p_k = D_\alpha f^k_\alpha
\quad
&\text{in}\,\,\Omega,
\\
\operatorname{div} u_k = g^k - (g^k)_\Omega
\quad
&\text{in}\,\,\Omega,
\\
u_k = 0
\quad
&\text{on} \,\, \partial \Omega.
\end{cases}
\end{equation}
Since $\Omega$ is bounded, $(u_k,p_k)\in \mathring W^1_q(\Omega)^d\times L_q(\Omega)$.
By the a priori estimate, we have
\begin{align*}
\|u_k\|_{W^1_q(\Omega)}+\|p_k\|_{L_q(\Omega)}
&\le N\|f^k\|_{L_q(\Omega)}+N\|g^k\|_{L_q(\Omega)}\\
&\le N\|f\|_{L_q(\Omega)}+N\|g\|_{L_q(\Omega)}.
\end{align*}
By the weak compactness, there is a subsequence  $(u_{k_j},p_{k_j})$, $u\in \mathring W^1_q(\Omega)^d$, and $p\in L_q(\Omega)$ satisfying $(p)_\Omega=0$, such that
$$
u_{k_j}\rightharpoonup u\quad\text{in}\,\,W^{1}_q(\Omega),\quad
p_{k_j}\rightharpoonup p\quad\text{in}\,\,L_q(\Omega).
$$
Taking the limit of the equations of $(u_{k_j},p_{k_j})$, it is easily seen that $(u,p)$ satisfies \eqref{eq10.49}.
In particular, $\operatorname{div}u = g$ in $\Omega$ because for any $\psi \in L_{q'}(\Omega)$, $1/q+1/q'=1$,
$$
\int_\Omega \left(g^k -(g^k)_\Omega\right) \psi \, dx = \int_\Omega (\operatorname{div} u_k) \psi \, dx \to
\int_\Omega (\operatorname{div} u) \psi \, dx
$$
and
$$
\int_\Omega \left( g^k - (g^k)_\Omega \right)  \psi \, dx \to \int_\Omega g \psi \, dx.
$$
The uniqueness follows from the a priori estimate \eqref{eq0328_02}. The theorem is proved.
\end{proof}

\section{Weighted case}
							\label{sec06}

We consider $\Omega$ and the operator $\cL$ with general coefficients $A^{\alpha\beta}$ satisfying Assumptions \ref{assum1004} and \ref{assum0711_1}.

\subsection{Mean oscillation estimates} In this subsection, we establish mean oscillation estimates for $D_{x'}u$ and $U$ (see \eqref{eq0829_01} for the definition of $U$) in the $L_q$ setting.

\begin{lemma}
							\label{lem0918_1}
Let $q \in (1,\infty)$, $R > 0$, $c > 1$, and $x_0 \in \bR^d$, and let the coefficients $\bar{A}^{\alpha\beta}(x_1)$ of the operator $\cL_0$ be measurable in $x_1$.
For any $f_\alpha, g \in L_q(B_{cR}(x_0))$, there exists $(u,p) \in W_{q}^1(B_R(x_0)) \times L_q(B_R(x_0))$ satisfying
\begin{equation*}
\begin{cases}
\cL_0 u + \nabla p = D_\alpha f_\alpha
\quad
&\text{in}\,\,B_R(x_0),
\\
\operatorname{div} u = g
\quad
&\text{in}\,\,B_R(x_0),
\end{cases}
\end{equation*}
and
\begin{equation*}
\|Du\|_{L_q(B_R(x_0))} + \|p\|_{L_q(B_R(x_0))} \le N \|f_\alpha\|_{L_q(B_{cR}(x_0))} + N \|g\|_{L_q(B_{cR}(x_0))},
\end{equation*}
where $N=N(d,\delta,c,q)$, but is independent of $R$.
\end{lemma}

\begin{proof}
See the proof of Lemma \ref{lem0918_2} below.
\end{proof}

\begin{lemma}
							\label{lem0918_2}
Let $q \in (1,\infty)$, $R > 0$, $c > 1$, and $x_0 \in \bR^d$, and let the coefficients $\bar{A}^{\alpha\beta}(x_1)$ of the operator $\cL_0$ be measurable in $x_1$.
For any $f_\alpha, g \in L_q(B_{cR}^+(x_0))$, there exists $(u,p) \in W_{q}^1(B_R^+(x_0)) \times L_q(B_R^+(x_0))$ satisfying
\begin{equation}
							\label{eq0918_03}
\begin{cases}
\cL_0 u + \nabla p = D_\alpha f_\alpha
\quad
&\text{in}\,\,B_R^+(x_0),
\\
\operatorname{div} u = g
\quad
&\text{in}\,\,B_R^+(x_0),
\\
u = 0 \quad
&\text{on} \,\, B_R(x_0) \cap \partial \bR^d_+,
\end{cases}
\end{equation}
and
\begin{equation}
							\label{eq0918_04}
\|Du\|_{L_q(B_R^+(x_0))} + \|p\|_{L_q(B_R^+(x_0))} \le N \|f_\alpha\|_{L_q(B_{cR}^+(x_0))} + N \|g\|_{L_q(B_{cR}^+(x_0))},
\end{equation}
where $N=N(d,\delta,c,q)$, but is independent of $R$.
\end{lemma}

\begin{proof}
Without loss of generality, we assume that $x_0 = 0$.
Set $\eta$ to be an infinitely differentiable function defined on $\bR^d$ such that
$$
\eta(x) = 1 \quad \text{on} \,\, B_1,
\quad
\eta(x) = 0 \quad \text{on} \,\, \bR^d \setminus B_{(c+1)/2}.
$$
Then, for $R>0$, set
$$
\cL_R u = D_\alpha (A^{\alpha\beta}_R D_\beta u ),
$$
where
$$
A^{\alpha\beta}_R(x) = \eta \bar{A}^{\alpha\beta}(R x_1) + (1-\eta) \Delta.
$$

Now, we fix a domain $\Omega$ with a smooth boundary such that
$$
B_{(c+1)/2}^+ \subset \Omega \subset B_c^+,
\quad
B_{(c+1)/2} \cap \partial \bR^d_+ \subset \partial \Omega.
$$
For this $\Omega$, find the constant $K_1 = K_1(d,c)$ in Assumption \ref{assum0224_1}.
Fix $\rho = \rho(d,\delta,c,q)$ in Theorem \ref{thm1} (also see Remark \ref{rem0918_1}).
We see that $A^{\alpha\beta}_R$ satisfy Assumption \ref{assum0711_1} ($\rho$) for $\Omega$ uniformly in $R$ if the $R_1$ in Assumption \ref{assum0711_1} is sufficiently small, and $R_0$ and $R_1$ in Theorem \ref{thm1} depend only on $c$.
By setting
$$
\hat{f}_\alpha(x) = R^{-1}f_\alpha(Rx),
\quad
\hat{g}(x) = R^{-1}g(Rx),
$$
and using Theorem \ref{thm1}, we find $(\hat{u}, \hat{p}) \in W_q^1(\Omega)^d \times L_q(\Omega)$ satisfying $(\hat{p})_{\Omega} = 0$,
$$
\begin{cases}
\cL_R \hat{u} + \nabla \hat{p} = D_\alpha \hat{f}_\alpha
\quad
&\text{in}\,\,\Omega,
\\
\operatorname{div} \hat{u} = \hat{g}
\quad
&\text{in}\,\,\Omega,
\\
\hat{u} = 0 \quad & \text{on} \,\, \partial \Omega,
\end{cases}
$$
and
$$
\|D\hat{u}\|_{L_q(\Omega)} + \|\hat{p}\|_{L_q(\Omega)} \le N \|\hat{f}_\alpha\|_{L_q(\Omega)} + N\|\hat{g}\|_{L_q(\Omega)},
$$
$$
\le N \|\hat{f}_\alpha\|_{L_q(B_c^+)} + N\|\hat{g}\|_{L_q(B_c^+)},
$$
where $N=N(d,\delta,c,q)$.
Then, we see that $(u,p) \in W_q^1(B_R^+)^d \times L_q(B_R^+)$, where
$$
u(x) = R^2 \, \hat{u}(x/R)
\quad\text{and}\quad
p(x) = R \, \hat{p}(x/R)
$$
satisfy \eqref{eq0918_03} and \eqref{eq0918_04}.
\end{proof}

\begin{lemma}
							\label{lem0311_1}
Let $q \in (1,\infty)$, $\mu, \nu \in (1,\infty)$, $1/\mu + 1/\nu = 1$, and $\kappa \ge 4$.
Then under Assumptions \ref{assum1004} and \ref{assum0711_1} ($\rho$),
for any $r \in (0,R_1/\kappa]$, $x_0 \in \bR^d$, and
$$
(u,p) \in W_{q\mu}^1(B_{\kappa r}(x_0))^d \times L_q(B_{\kappa r}(x_0))
$$
satisfying
$$
\begin{cases}
\cL u + \nabla p = D_\alpha f_\alpha
\quad
&\text{in}\,\,B_{\kappa r}(x_0),
\\
\operatorname{div} u = g
\quad
&\text{in}\,\,B_{\kappa r}(x_0),
\end{cases}
$$
where $f_\alpha, g \in L_q(B_{\kappa r}(x_0))$, there exists a $(d^2+1)$-dimensional vector-valued function $\cU$ on $B_{\kappa r}(x_0)$ such that, on $B_{\kappa r}(x_0)$,
\begin{equation}
							\label{eq10.54}
N^{-1}(|Du|+|p|) \le |\cU| \le N (|Du|+|p|)
\end{equation}
and
\begin{multline*}
\left( |\cU - (\cU)_{B_r(x_0)}|\right)_{B_r(x_0)} \le N \kappa^{-\frac{1}{2}} \left( |Du|^q \right)_{B_{\kappa r}(x_0)}^{\frac{1}{q}}
\\
+ N \kappa^{\frac d q} \rho^{\frac{1}{q\nu}} \left( |Du|^{q\mu} \right)_{B_{\kappa r}(x_0)}^{\frac{1}{q\mu}} + N \kappa^{\frac d q} \left(|f_\alpha|^q\right)_{B_{\kappa r}(x_0)}^{\frac{1}{q}}+N \kappa^{\frac d q} \left(|g|^q\right)_{B_{\kappa r}(x_0)}^{\frac{1}{q}},
\end{multline*}
where $N=N(d,\delta,\mu, q)$.
\end{lemma}

\begin{proof}
We assume that $x_0 = 0$.
By Assumption \ref{assum0711_1}, there is a coordinate system such that
\begin{equation}
							\label{eq0918_05}
\dashint_{B_{\kappa r}} |A^{\alpha\beta}(x_1,x') - \bar{A}^{\alpha\beta}(x_1) | \, dx \le \rho,
\end{equation}
where
\begin{equation*}
\bar{A}^{\alpha\beta}(x_1) = \dashint_{B'_{\kappa r}} A^{\alpha\beta}(x_1,x') \, dx'.
\end{equation*}
Let $\cL_0$ be the elliptic operator with the coefficients $\bar{A}^{\alpha\beta}$.
By Lemma \ref{lem0918_1}, there exists $(w,p_1) \in W_q^1(B_{\kappa r/2})^d \times L_q(B_{\kappa r/2})$ such that
$$
\begin{cases}
\cL_0 w + \nabla p_1 =D_\alpha \bar{f}_\alpha
\quad
&\text{in}\,\,B_{\kappa r/2},
\\
\operatorname{div} w = g - (g)_{B_{\kappa r/2}}
\quad
&\text{in}\,\,B_{\kappa r/2},
\end{cases}
$$
where $\bar{f}_\alpha = f_\alpha + (\bar{A}^{\alpha\beta} - A^{\alpha\beta}) D_\beta u$
and
\begin{equation}
							\label{eq0225_02}
\|Dw\|_{L_q(B_{\kappa r/2})} + \|p_1\|_{L_q(B_{\kappa r/2})} \le N \|\bar{f}_\alpha\|_{L_q(B_{\kappa r})} + N \|g\|_{L_q(B_{\kappa r})}.
\end{equation}
This estimate implies
\begin{equation}
							\label{eq0302_04}
\|Dw\|_{L_q(B_r)} + \|p_1\|_{L_q(B_r)} \le N \left( \|\bar{f}_\alpha\|_{L_q(B_{\kappa r})} + \|g\|_{L_q(B_{\kappa r})}\right),
\end{equation}
where $N=N(d,\delta,q)$.
Note that by the boundedness of $A^{\alpha\beta}$, H\"older's inequality, and \eqref{eq0918_05}, we have
\begin{equation}
							\label{eq0918_06}
\|\bar{f}_\alpha\|_{L_q(B_{\kappa r})} \le \|f_\alpha\|_{L_q(B_{\kappa r})} + N \rho^{1/(q\nu)} (\kappa r)^{d/ (q\nu)} \|Du\|_{L_{q\mu}(B_{\kappa r})},
\end{equation}
where $N=N(\delta,\mu,q)$.

Set $(v,p_2) = (u,p) - (w,p_1)$, which satisfies
$$
\begin{cases}
\cL_0 v + \nabla p_2 = 0
\quad
&\text{in}\,\,B_{\kappa r/2},
\\
\operatorname{div} v = (g)_{B_{\kappa r/2}}
\quad
&\text{in}\,\,B_{\kappa r/2}.
\end{cases}
$$
Then, by Lemma \ref{lem0829_3},
\begin{align}
&\left( |D_{x'} v - (D_{x'} v)_{B_r}| \right)_{B_r} + \left( |V - (V)_{B_r}| \right)_{B_r}\nonumber
\\
							\label{eq10.17}
&\le (2r)^{1/2} \left[D_{x'}v\right]_{C^{1/2}(B_r)} + (2r)^{1/2} \left[V\right]_{C^{1/2}(B_r)}
\le N \kappa^{-1/2}\left( |Dv|^q \right)^{1/q}_{B_{\kappa r/2}},
\end{align}
where
$$
V_1 = \sum_{j=1}^d \sum_{\beta=1}^d \bar{A}_{1j}^{1\beta} D_\beta v_j + p_2,
\quad
V_i = \sum_{j=1}^d\sum_{\beta = 1}^d \bar{A}_{ij}^{1\beta}D_\beta v_j,
\quad
i = 2, \ldots, d,
$$
and $N=N(d,\delta)$.
Set
\begin{equation}
                                            \label{eq10.40}
\cU=(D_{x'}u,\operatorname{div} u, U_1,\ldots,U_d),
\end{equation}
where
\begin{equation*}
U_1 = \sum_{j=1}^d \sum_{\beta=1}^d \bar{A}_{1j}^{1\beta} D_\beta u_j + p,\quad
U_i = \sum_{j=1}^d\sum_{\beta = 1}^d \bar{A}_{ij}^{1\beta}D_\beta u_j,
\quad
i = 2, \ldots, d.
\end{equation*}
Then, it follows from the triangle inequality, \eqref{eq10.17}, and H\"older's inequality that
\begin{align*}
&\left( |\cU - (\cU)_{B_r}| \right)_{B_r} \le N \left( |D_{x'} v - (D_{x'} v)_{B_r}| \right)_{B_r}
\\
& + N \left( |V - (V)_{B_r}| \right)_{B_r} + \left( |g| \right)_{B_r} + N \left( |D w| + |p_1|\right)_{B_r}\\
&\le N \kappa^{-1/2} \left( |Dv|^q \right)_{B_{\kappa r/2}}^{1/q}
+N \kappa^{d/q} \left( |g|^q \right)_{B_{\kappa r/2}}^{1/q}
+ N \left( |D w| + |p_1| \right)_{B_r},
\end{align*}
where $N=N(d,\delta)$.
Together with the estimates \eqref{eq0225_02}, \eqref{eq0302_04}, \eqref{eq0918_06}, and the fact that $u = v+w$, this shows that
\begin{align*}
&\left( |\cU - (\cU)_{B_r}| \right)_{B_r} \le N \kappa^{-1/2} \left( |Du|^q \right)_{B_{\kappa r}}^{1/q}
\\
& + N \kappa^{d/q} \rho^{1/(q\nu)} \left( |Du|^{q\mu} \right)_{B_{\kappa r}}^{1/(q\mu)} + N \kappa^{d/q} \left(|f_\alpha|^q + |g|^q \right)_{B_{\kappa r}}^{1/q}.
\end{align*}
Finally, it is easy to check \eqref{eq10.54} using the definition of $U$.
The lemma is proved.
\end{proof}

Recall 
$$
\Omega_r(x_0) = \Omega \cap B_r(x_0).
$$

\begin{lemma}
                        \label{lem6.4}
Let $q \in (1,\infty)$, $\mu, \nu \in (1,\infty)$, $1/\mu + 1/\nu = 1$, and $\kappa \ge 64$.
Then, under Assumptions \ref{assum1004} and \ref{assum0711_1} ($\rho$) such that $\rho \kappa \le 1/4$, for any $r \in (0,R_1/\kappa]$, $x_0 \in \overline{\Omega}$, and
$$
(u,p) \in W_{q\mu}^1(\Omega_{\kappa r}(x_0))^d \times L_{q\mu}(\Omega_{\kappa r}(x_0))
$$
satisfying
\begin{equation}
							\label{eq11.33}
\begin{cases}
\cL u + \nabla p = D_\alpha f_\alpha
\quad
&\text{in}\,\,\Omega_{\kappa r}(x_0),
\\
\operatorname{div} u = g
\quad
&\text{in}\,\,\Omega_{\kappa r}(x_0),
\\
u = 0
\quad
&\text{on} \,\, \partial \Omega \cap B_{\kappa r}(x_0),
\end{cases}
\end{equation}
where $f_\alpha \in L_q(\Omega_{\kappa r}(x_0))$, there exists a $(d^2+1)$-dimensional vector-valued function $\cU$ on $\Omega_{\kappa r}(x_0)$ such that, on $\Omega_{\kappa r}(x_0)$,
\begin{equation}
							\label{eq0919_01}
N^{-1}(|Du|+|p|) \le |\cU| \le N (|Du|+|p|),
\end{equation}
and
\begin{align}
							\label{eq3.33}	
&\left( |\cU - (\cU)_{\Omega_r(x_0)}|\right)_{\Omega_r(x_0)}
\le N (\kappa^{-\frac 1 2}+\kappa\rho) \left( |Du|^q + |p|^q \right)_{\Omega_{\kappa r}(x_0)}^{\frac{1}{q}}
\nonumber\\
&\quad + N \kappa^{\frac d q} \rho^{\frac{1}{q\nu}} \left( |Du|^{q\mu} +|p|^{q\mu} \right)_{\Omega_{\kappa r}(x_0)}^{\frac{1}{q\mu}}
+ N \kappa^{\frac d q} \left(|f_\alpha|^q\right)_{\Omega_{\kappa r}(x_0)}^{\frac{1}{q}}+N \kappa^{\frac d q} \left(|g|^q\right)_{\Omega_{\kappa r}(x_0)}^{\frac{1}{q}},
\end{align}
where $N=N(d,\delta,\mu, q)$.
\end{lemma}

\begin{proof}
Let $\tilde x\in \partial\Omega$ be such that $|x_0 -\tilde x|=\text{dist}(x_0,\partial\Omega)$.
We consider two cases.

\noindent{\bf Case 1: $|x_0 -\tilde x|\ge \kappa r/16$.} In this case, we have 
$$
\Omega_r(x_0)=B_r(x_0)\subset B_{\kappa r/16}(x_0)\subset \Omega.
$$
Since $\kappa/16\ge 4$, \eqref{eq3.33} follows from Lemma \ref{lem0311_1}.

\noindent{\bf Case 2: $|x_0 -\tilde x|< \kappa r/16$.} In this case, the proof is similar to that of Proposition \ref{prop0719_1}. Without loss of generality, one may assume that $\tilde x$ is the origin. Note that
\begin{equation*}
\Omega_r(x_0)\subset \Omega_{\kappa r/4} \subset \Omega_{\kappa r/2}
\subset \Omega_{\kappa r}(x_0).
\end{equation*}
Denote $R=\kappa r/2 \, (\le R_1/2)$.
Due to Assumption \ref{assum0711_1}, we can perform an orthogonal transformation to obtain
$$
 \{(x_1, x'):\rho R< x_1\}\cap B_R
 \subset\Omega\cap B_R
 \subset \{(x_1, x'):-\rho R<x_1\}\cap B_R
$$
and
\begin{equation}
                                \label{eq17_50}
\dashint_{B_R} \left| A^{\alpha\beta}(x_1,x') - \bar A^{\alpha\beta}(x_1)\right| \, dx \le \rho,
\end{equation}
where
\begin{equation}
							\label{eq0715_001}
\bar A^{\alpha\beta}(x_1)=\dashint_{B'_R} A^{\alpha\beta}(x_1,x')\,dx'.
\end{equation}
Take a smooth function $\chi$ on $\bR$ such that
$$
\chi(x_1)\equiv 0\quad\text{for}\,\,x_1\le \rho R,
\quad \chi(x_1)\equiv 1\quad\text{for}\,\,x_1\ge  2\rho R,\quad
\text{and}
\quad
|\chi'| \le 2(\rho R)^{-1}.
$$
Denote $\cL_0$ to be the elliptic operator with the coefficients $\bar A^{\alpha\beta}$ from \eqref{eq0715_001}.
Let $\hat u=\chi u$, which vanishes on $B_R\cap \{x_1\le \rho R\}$. From \eqref{eq11.33}, it is easily seen that $(\hat u,p)$ satisfies
\begin{equation}
                                    \label{eq17.23b0}
\begin{cases}
\cL_0 \hat u+\nabla p = D_\alpha (\tilde f_\alpha+h_\alpha)
\quad
&\text{in}\,\,B_{R}\cap\{x_1>\rho R\},
\\
\operatorname{div}  \hat u = \chi g+\chi' u_1
\quad
&\text{in}\,\,B_{R}\cap\{x_1>\rho R\},
\\
\hat u = 0
\quad
&\text{on} \,\, B_{R}\cap \{x_1=\rho R\},
\end{cases}
\end{equation}
where
$$
\tilde f_\alpha=f_\alpha+(\bar A^{\alpha\beta}- A^{\alpha\beta})D_\beta  u\quad
\text{and}
\quad
h_\alpha=\bar A^{\alpha\beta} D^{\beta}((\chi-1) u).
$$
For $\tau \in [0,\infty)$, set
$$
\widetilde{B}_r^+(\tau,0) = B_r(\tau,0) \cap \{x_1 > \tau\},
$$
where $0 \in \bR^{d-1}$.
Since $\rho\in (0,1/16)$, we have
$$
\Omega_{R/2} \subset \Omega_{3R/4}(\rho R,0)\quad
\text{and}
\quad
\widetilde{B}^+_{3R/4}(\rho R,0) \subset \widetilde{B}^+_{7R/8}(\rho R,0) \subset B_{R}\cap\{x_1>\rho R\}.
$$
By Lemma \ref{lem0918_2}, there exists
$$
(\hat w, \hat{p}_1)\in W^{1}_q(\widetilde{B}^+_{3R/4}(\rho R,0))^d \times L_q(\widetilde{B}^+_{3R/4}(\rho R,0))
$$
satisfying
\begin{equation}
							\label{eq0714_003}
\begin{cases}
\cL_0 \hat w+\nabla \hat{p}_1 = D_\alpha (\tilde f_\alpha+h_\alpha)
\quad
&\text{in}\,\,\widetilde{B}^+_{3R/4}(\rho R,0),
\\
\operatorname{div} \hat w = \chi g+\chi' u_1-\big(\chi g+\chi' u_1\big)_{\widetilde{B}^+_{3R/4}(\rho R,0)}
\quad
&\text{in}\,\,\widetilde{B}^+_{3R/4}(\rho R,0),
\\
\hat w = 0
\quad
&\text{on} \,\, B_{3R/4}(\rho R,0) \cap \{x_1 = \rho R\},
\end{cases}
\end{equation}
and
\begin{align}
                                \label{eq3.001}
&\|D\hat w\|_{L_q(\widetilde{B}^+_{3R/4}(\rho R,0))} + \|\hat{p}_1\|_{L_q(\widetilde{B}^+_{3R/4}(\rho R,0))}
\nonumber\\
&\le N\|\tilde f_\alpha\|_{L_q(\cD)}
+N\|h_\alpha\|_{L_q(\cD)}
+N\|\chi g+\chi' u\|_{L_q(\cD)}
\nonumber\\
&\le N\|f_\alpha\|_{L_q(\cD)}
+N\|(\bar A^{\alpha\beta}- A^{\alpha\beta})D_\beta  u\|_{L_q(\cD)}
\nonumber\\
&\quad +N\|D((\chi-1)u)\|_{L_q(\cD)}
+N\|g\|_{L_q(\cD)}+N\|\chi' u\|_{L_q(\cD)},
\end{align}
where $\cD = \widetilde{B}^+_{7R/8}(\rho R,0)$ and $N = N(d,\delta,q)$.
Using the fact that $|A^{\alpha\beta}| \le \delta^{-1}$ and $\cD \subset B_R$ together with \eqref{eq17_50} and H\"older's inequality, it follows that
\begin{equation}
							\label{eq0329_01}
\|(\bar A^{\alpha\beta}- A^{\alpha\beta})D_\beta  u\|_{L_q(\cD)}\le N\rho^{\frac 1 {q\nu}}R^{\frac d {q\nu}}
\|Du\|_{L_{q\mu}(\cD)}.
\end{equation}
Since $\chi-1$ is supported on $\{x_1\le 2\rho R\}$, H\"older's inequality implies that
\begin{equation}
							\label{eq0328_006}
\|(\chi-1)Du\|_{L_q(\cD)} \le N\rho^{\frac 1 {q\nu}}R^{\frac d {q\nu}}
\|Du\|_{L_{q\mu}(\Omega_{R})}.
\end{equation}
Using H\"older's inequality again, together with the fact that $\chi'$ is supported on $\{\rho R \le x_1 \le 2\rho R\}$, we have
\begin{equation}
							\label{eq0328_007}
\|\chi' u\|_{L_q(\cD)}
\le N \rho^{\frac{1}{q\nu}} R^{\frac{d}{q\nu}}
\|\chi' u\|_{L_{q\mu}(\cD)} \le N \rho^{\frac{1}{q\nu}} R^{\frac{d}{q\nu}} \|Du\|_{L_{q\mu}(\Omega_R)}.
\end{equation}
Note that in the last inequality above we used Hardy's inequality, the boundary condition $u = 0$ on $\partial\Omega$, and the observation that
$$
|\chi'| \le N(x_1 - \phi(x'))^{-1}
$$
for $(x_1,x') \in \Omega_R$, where $\phi(x')$ is the largest number such that $(\phi(x'),x') \in \partial \Omega$.
The inequalities \eqref{eq0329_01}, \eqref{eq0328_006}, and \eqref{eq0328_007}, together with \eqref{eq3.001}, imply that
\begin{equation}
            \label{eq21.52h0}
(|D \hat w|^q + |\hat{p}_1|^q)_{\widetilde{B}^+_{3R/4}(\rho R,0)}^{\frac 1 q}
\le N\rho^{\frac 1 {q\nu}} (|Du|^{q\mu})_{\Omega_R}^{\frac 1 {q\mu}}+
N(|f_\alpha|^q+|g|^q)_{\Omega_{R}}^{\frac 1 q}.
\end{equation}
We extend $\hat w$ to be zero  in $\Omega_{3R/4}(\rho R,0)\cap \{x_1<\rho R\}$, so that $\hat w\in W^{1}_2(\Omega_{3R/4}(\rho R,0))$, and we let
$$
w=\hat w+(1-\chi)u.
$$
We also set
$$
p_1 =
\begin{cases}
\hat{p}_1 \quad &\text{in} \,\,\tilde{B}^+_{3R/4}(\rho R,0),
\\
p \quad &\text{in} \,\, \Omega_{3R/4}(\rho R, 0) \cap \{x_1 \rho R\}.
\end{cases}
$$
By the same reasoning as in \eqref{eq0328_006} and \eqref{eq0328_007}, we have 
$$
\|D\left((1-\chi) u\right)\|_{L_q(\Omega_{3R/4}(\rho R,0))} \le N \rho^{\frac{1}{q\nu}} R^{\frac{d}{q \nu}} \|Du\|_{L_{q\mu}(\Omega_R)},
$$
and
$$
\|p_1\|_{L_q(\Omega_{3R/4}(\rho R,0))} \le \|\hat{p}_1\|_{L_q(\tilde{B}^+_{3R/4}(\rho R,0))} + N \rho^{\frac{1}{q\nu}} R^{\frac{d}{q\nu}} \|p\|_{L_{q\mu}(\Omega_R)}.
$$
From these inequalities and \eqref{eq21.52h0}, we deduce that
\begin{equation}
            \label{eq18.340h}
(|Dw|^q + |p_1|^q)_{\Omega_{3R/4}(\rho R,0)}^{\frac 1 q}
\le N\rho^{\frac 1 {q\nu}} (|Du|^{q\mu} + |p|^{q\mu})_{\Omega_R}^{\frac 1  {q\mu}}+
N(|f_\alpha|^q+|g|^q)_{\Omega_R}^{\frac 1 q}.
\end{equation}
Note that, because $\kappa \rho\le 1/4$, it holds that
$$
\Omega_r(x_0) \subset \Omega_{3R/4}(\rho R,0)\quad \text{and} \quad |\Omega_{3R/4}(\rho R,0)|/|\Omega_r(x_0)|
\le N(d) \kappa^{d}.
$$
Thus, from \eqref{eq18.340h} we also obtain that
\begin{equation}
                                \label{eq28_01}
(|Dw|^q + |p_1|^q)_{\Omega_r(x_0)}^{\frac 1 q}
\le N\kappa^{\frac d q}\rho^{\frac 1 {q\nu}} (|Du|^{q\mu} + |p|^{q\mu})_{\Omega_R}^{\frac 1 {q\mu}}+N \kappa^{\frac d q}
(|f_\alpha|^q+|g|^q)_{\Omega_R}^{\frac 1 q}.
\end{equation}

Next, we set $v= u- w$ $(=\chi u-\hat w)$ and $p_2=p-p_1$ in $\Omega_{3R/4}(\rho R,0)$.
From \eqref{eq17.23b0} and \eqref{eq0714_003}, it is easily seen that $(v,p_2)$ satisfies
\begin{equation*}
\begin{cases}
\cL_0 v+\nabla p_2 = 0
\quad
&\text{in}\,\,\widetilde{B}^+_{3R/4}(\rho R,0),
\\
\operatorname{div} v = \big(\chi g+\chi' u_1\big)_{\widetilde{B}^+_{3R/4}(\rho R,0)}
\quad
&\text{in}\,\,\widetilde{B}^+_{3R/4}(\rho R,0),
\\
v = 0
\quad
&\text{on} \,\, B_{3R/4}(\rho R,0)\cap \{x_1=\rho R\}.
\end{cases}
\end{equation*}
Denote
$$
\cD_1=\Omega_{r}(x_0)\cap \{x_1\le \rho R\},
\quad
\cD_2=\Omega_{r}(x_0)\cap \{x_1> \rho R\},
\quad\text{and}\,\,
\cD_3=\widetilde{B}^+_{R/4}(\rho R,0).
$$
We see that $\cD_2\subset \cD_3$ and $|\cD_1|\le N\kappa\rho|\Omega_{r}(x_0)|$, where the latter follows from the fact that $\cD_1 = \Omega_r(x_0) \cap \{ - \rho R \le x_1 \le \rho R\}$.
We set
$$
V_1 = \sum_{j=1}^d \sum_{\beta=1}^d \bar{A}_{1j}^{1\beta} D_\beta v_j + p_2,
\quad
V_i = \sum_{j=1}^d\sum_{\beta = 1}^d \bar{A}_{ij}^{1\beta}D_\beta v_j,
\quad
i = 2, \ldots, d,
$$
where the coefficients $\bar{A}^{1\beta}(x_1)$ are taken from \eqref{eq0715_001}.
Note that $v=V=0$ in $\cD_1$.
Then, by applying Lemma \ref{lem0829_3},
we get 
\begin{align}
							\label{eq5.01}
&\big(|V-(V)_{\Omega_r(x_0)}|\big)_{\Omega_r(x_0)}
+\big(|D_{x'}v-(D_{x'}v)_{\Omega_r(x_0)}|\big)_{\Omega_r(x_0)}
\nonumber\\
&\le N r^{\frac 1 2}\big([V]_{C^{1/2}(\cD_2)}
+[D_{x'}v]_{C^{1/2}( \cD_2)}\big)
+N\kappa\rho \left(\|V\|_{L_\infty(\cD_2)} +\|D_{x'}v\|_{L_\infty(\cD_2)} \right)
\nonumber\\
&\le N r^{\frac 1 2}\big([V]_{C^{1/2}(\cD_3)}
+[D_{x'}v]_{C^{1/2}( \cD_3)}\big)
+N\kappa\rho \left(\|V\|_{L_\infty(\cD_3)} +\|D_{x'}v\|_{L_\infty(\cD_3)} \right)
\nonumber\\
&\le N(\kappa^{- \frac 1 2}+\kappa \rho)(|Dv|^q + |p_2|^q)_{\tilde{B}^+_{R/2}(\rho R,0)}^{\frac 1 q}.
\end{align}

Now, we define $\cU=(D_{x'}u,\operatorname{div} u, U_1, U_2,\ldots,U_d)$ as in \eqref{eq10.40}.
Note that $\cU$ satisfies \eqref{eq0919_01}.
From the triangle inequality and \eqref{eq5.01}, we have 
\begin{multline*}
\big(|\cU-(\cU)_{\Omega_r(x_0)}|\big)_{\Omega_r(x_0)} \le N\big(|D_{x'}v-(D_{x'}v)_{\Omega_r(x_0)}|\big)_{\Omega_r(x_0)}
\\
+ N \big(|V-(V)_{\Omega_r(x_0)}|\big)_{\Omega_r(x_0)}
+N(|g| + |Dw| + |p_1|)_{\Omega_r(x_0)}
\\
\le N(\kappa^{-\frac 1 2}+\kappa \rho)(|Dv|^q + |p_2|^q)_{\tilde{B}^+_{R/2}(\rho R, 0)}^{\frac 1 q}
+N(|g|^q  + |Dw|^q + |p_1|^q)^{\frac 1 q}_{\Omega_r(x_0)}.
\end{multline*}
To estimate the terms on the right-hand side of the last inequality above, we note the following.
First, by the fact that $u= w+v$ and $p=p_1+p_2$, and  \eqref{eq18.340h} we have 
$$
( |Dv|^q + |p_2|^q)^{\frac 1 q}_{\tilde{B}^+_{R/2}(\rho R, 0)} \le ( |Du|^q + |p|^q )^{\frac 1 q}_{\tilde{B}^+_{R/2}(\rho R, 0)} + ( |Dw|^q + |p_1|^q)^{\frac 1 q}_{\tilde{B}^+_{R/2}(\rho R, 0)}
$$
$$
\le N( |Du|^q + |p|^q )^{\frac 1 q}_{\Omega_R} + N\rho^{\frac 1 {q\nu}} (|Du|^{q\mu} + |p|^{q\mu})_{\Omega_R}^{\frac 1  {q\mu}}+
N(|f_\alpha|^q+|g|^q)_{\Omega_R}^{\frac 1 q}.
$$
Now, we use \eqref{eq28_01} and the fact that $|\Omega_R| \le N \kappa^d |\Omega_r(x_0)|$ by the condition $\kappa \rho \le 1/4$ to obtain 
$$
(|g|^q + |Dw|^q + |p_1|^q)^{\frac 1 q}_{\Omega_r(x_0)} \le N \kappa^{\frac d q} \rho^{\frac 1 {q\nu}} (|Du|^{q\mu} + |p|^{q\mu})_{\Omega_R}^{\frac 1 {q\mu}} + N \kappa^{\frac d q} \left( |f_\alpha|^q + |g|^q \right)_{\Omega_R}^{\frac 1 q}.
$$
Combining the inequalities above, we get
$$
\big(|\cU-(\cU)_{\Omega_r(x_0)}|\big)_{\Omega_r(x_0)} \le N(\kappa^{- \frac 1 2} + \kappa \rho) (|Du|^q + |p|^q)^{\frac 1 q}_{\Omega_R}
$$
$$
+ N \kappa^{\frac d q} \rho^{\frac 1 {q\nu}} (|Du|^{q\mu} + |p|^{q\mu})^{\frac 1 {q\mu}}_{\Omega_R}
+ N \kappa^{\frac d q} \left( |f_\alpha|^q + |g|^q \right)^{\frac 1 q}_{\Omega_R}.
$$
Since $\Omega_R = \Omega_{\kappa r/2} \subset \Omega_{\kappa r}(x_0)$ and the volumes of these two sets are comparable, we finally obtain \eqref{eq3.33}. The lemma is proved.
\end{proof}

\subsection{Proofs of Theorems \ref{thm3} and \ref{thm2}}

First we note that a bounded Reifenberg flat domain is a space of homogeneous type, see \cite[Remark 7.3]{DK15}, which is endowed with the Euclidean distance and a doubling measure $\mu$ that is naturally inherited from the Lebesgue measure. From a result in Christ \cite[Theorem 11]{MR1096400}, there exists a filtration of partitions of $\Omega$ in the following sense. For each $n\in\bZ$, there exists a collection of disjoint open subsets $\bC_n:=\{Q_\alpha^n\,:\,\alpha\in I_n\}$ for some index set $I_n$, satisfying the following properties:
\begin{enumerate}
\item[(1)] For any $n\in \bZ$, $\mu(\Omega\setminus \bigcup_{\alpha}Q_\alpha^n)=0$;
\item[(2)] For each $n$ and $\alpha\in I_n$, there exists a unique $\beta\in I_{n-1}$ such that $Q_\alpha^n\subset Q_\beta^{n-1}$;
\item[(3)] For each $n$ and $\alpha\in I_n$, $\text{diam}(Q_\alpha^n)\le N_0\delta_0^n$;
\item[(4)] Each $Q_\alpha^n$ contains some ball $B_{\varepsilon_0\delta_0^n}(z_\alpha^n)$;
\end{enumerate}
for some constants $\delta_0\in (0,1)$, $\varepsilon_0>0$, and $N_0$ depending only on $d$, $R_0$, and $K$.

We also use the following filtration of partitions of $\bR^d$:
$$
\bC_n := \{ C_n = C_n(i_1, \ldots, i_d): (i_1, \ldots, i_d) \in \bZ^d \},
$$
where $n \in \bZ$ and
$$
C_n(i_1, \ldots, i_d) = [i_1 2^{-n}, (i_1 + 1) 2^{-n}).
$$
For a filtration of partitions of $\bR^d_+$, we replace $i_1 \in \bZ$ by $i_1 \in \{0,1,2,\ldots,\}$.

The following lemma is proved in \cite[Theorem 2.4]{DK15}, which is in the spirit of \cite[Theorem 2.7]{MR2540989}.

\begin{lemma}							\label{lem6.5}
Let $q \in (1, \infty)$, $\omega\in A_q$, and $\Omega$ be either $\bR^d$, $\bR^{d}_+$, or a bounded Reifenberg flat domain as in Theorem \ref{thm2}.
Suppose that
$$
\ff,\fg,\fh\in L_{q,\omega}(\Omega),\quad |\ff| \le \fh,
$$
and for each $n \in \bZ$ and $Q \in \bC_n$,
there exists a measurable function $\ff^Q$ on $Q$
such that $|\ff| \le \ff^Q \le \fh$ on $Q$ and
\begin{equation}							 \label{eq10.53}
\dashint_Q |\ff^Q(x) - \left(\ff^Q\right)_Q| \,dx
\le N_0\fg(y)\quad \forall \,y\in Q
\end{equation}
for some constant $N_0>0$.

(i) When $\Omega=\bR^d$ or $\bR^d_+$, we have
$$
\|\ff\|^{q}_{L_{q,\omega}}\le NN_0\|\fg\|^\beta_{L_{q,\omega}}\|\fh \|_{L_{q,\omega}}^{q-\beta},
$$
where $\beta\in (0,1]$ and  $N>0$ are constants depending only on $d$, $q$, and $[\omega]_{A_q}$.

(ii) Otherwise, we have
\begin{equation*}
\|\ff\|^{q}_{L_{q,\omega}}\le NN_0\|\fg\|^\beta_{L_{q,\omega}}\|\fh \|_{L_{q,\omega}}^{q-\beta}+
N_1\|\fh\|^{q}_{L_1},
\end{equation*}
where $\beta\in (0,1]$ and  $N>0$ are constants depending only on $d$, $q$, $R_0$, $K$, and $[\omega]_{A_q}$, and $N_1$ depends only on $\omega(\Omega)$ and the same parameters as $N$.
\end{lemma}

\begin{proof}[Proof of Theorem \ref{thm3}]
We only treat the case when $\Omega = \bR^d_+$, as the whole space case is similar.
By the reverse H\"older's inequality for $A_q$ weights, we have $\omega\in A_{\tilde q}$ for some $\tilde q\in (1,q)$ which depends on $d$, $q$, and $[\omega]_{A_q}$. Let $q_0=q/\tilde q\in (1,q)$. Then clearly, for any function $f\in L_{q,\omega}(\Omega)$, we have $|f|^{q_0}\in L_{\tilde q,\omega}(\Omega)$, and by Lemma 3.1 of \cite{DK15S}, $f\in L_{q_0,\text{loc}}(\Omega)$.

For $x \in \bR^d_+$ and $C_n \in \bC_n$ such that $x \in C_n$, find $x_0 \in \bR^d_+$ and the smallest $r \in (0,\infty)$ (indeed, $r = 2^{-n-1}\sqrt{d}$) satisfying
$C_n \subset B_r(x_0)$ and
\begin{equation}
							\label{eq0311_03}
\dashint_{C_n} | h(x) - (h)_{C_n} | \, dx \le N(d) \dashint_{B_r(x_0)}| h(x) - (h)_{B_r(x_0)} | \, dx.
\end{equation}
Under the conditions of the theorem, Assumption \ref{assum0711_1} is satisfied with $\rho=0$ and $R_1 = \infty$.
Since $(u,p) \in W_{q_0}^1(\Omega_{\kappa r}^+(x_0))^d \times L_{q_0}(B_{\kappa r}^+(x_0))$, it follows from Lemma \ref{lem6.4} that for any $\kappa \ge 64$ and $r>0$, \eqref{eq3.33} holds with $q$ replaced by $q_0$ and $\rho=0$, where $\cU$ is a function defined on $C_n$ and satisfies \eqref{eq0919_01}.
Clearly, each term in the right-hand side of \eqref{eq3.33} is bounded by its maximal function at $x$.
From this and \eqref{eq0311_03}, we have 
\begin{equation*}
\left( |\cU - (\cU)_{C_n}| \right)_{C_n}
\le N \kappa^{-\frac 1 2} \left( \cM (|Du|^{q_0}+|p|^{q_0}) (x) \right)^{\frac 1{q_0}}
+ N \kappa^{\frac d {q_0}}
\left( \cM (F^{q_0})(x) \right)^{\frac 1 {q_0}}
\end{equation*}
for $x \in C_n$ and $\kappa \ge 64$, where $N=N(d,\delta,q_0)$ and $F=|f_\alpha|+|g|$.
Recall the inequality \eqref{eq0919_01}.
Applying Lemma 6.5 (i) with $\ff=N^{-1}(|Du|+|p|)$, $\fh=N(|Du|+|p|)$, and
$$
\fg=\kappa^{-\frac 1 2} \left( \cM (|Du|^{q_0}+|p|^{q_0}) (x) \right)^{\frac 1{q_0}}
+ \kappa^{\frac d {q_0}}
\left( \cM (F^{q_0})(x) \right)^{\frac 1 {q_0}},
$$
and using the maximal function theorem with $A_p$ weights (see, for instance, \cite{MR0740173}) on the above pointwise estimate, we obtain 
\begin{align*}
&\|Du\|_{L_{q,\omega}} + \|p\|_{L_{q,\omega}}\\
&\le N \kappa^{-\frac 1 2} \left\|\left( \cM (|Du|^{q_0}+|p|^{q_0}) \right)^{\frac 1 {q_0}}\right\|_{L_{q,\omega}}+ N \kappa^{\frac d {q_0}}
\left\|\left( \cM (F^{q_0}) \right)^{\frac 1 {q_0}}\right\|_{L_{q,\omega}}\\
&= N \kappa^{-\frac 1 2} \left\|\cM (|Du|^{q_0}+|p|^{q_0})  \right\|_{L_{\tilde q,\omega}}^{\frac 1 {q_0}}+ N \kappa^{\frac d {q_0}}
\left\|\cM (F^{q_0}) \right\|_{L_{\tilde q,\omega}}^{\frac 1 {q_0}}\\
&\le N \kappa^{-\frac 1 2} \left\||Du|^{q_0}+|p|^{q_0}\right\|_{L_{\tilde q,\omega}}^{\frac 1 {q_0}}+ N \kappa^{\frac d {q_0}}
\left\|F^{q_0}\right\|_{L_{\tilde q,\omega}}^{\frac 1 {q_0}}\\
&\le N \kappa^{-\frac 1 2} \left(\|Du\|_{L_{q,\omega}}+\|p\|_{L_{q,\omega}}\right) + N \kappa^{\frac d {q_0}} \left( \|f_\alpha\|_{L_{q,\omega}} + \|g\|_{L_{q,\omega}} \right),
\end{align*}
where $L_{q,\omega} = L_{q,\omega}(\bR^d_+)$ and $N = N(d,\delta,q,[\omega]_{A_q})$.
Upon taking a sufficiently large $\kappa \ge 64$, which depends only on $d$, $\delta$, $q$, and $[\omega]_{A_q}$ such that $N \kappa^{-1/2} \le 1/2$, we arrive at
\eqref{eq0307_03}. The theorem is proved.
\end{proof}

\begin{proof}[Proof of Theorem \ref{thm2}]
As in the proof of Theorem \ref{thm3}, we have $\omega\in A_{\tilde q}$ for some $\tilde q\in (1,q)$ which depends on $d$, $q$, $R_0$, $K$, and $[\omega]_{A_q}$. Let $q_0=\mu=(q/\tilde q)^{1/2}\in (1,q^{1/2})$. Then, for any function $f\in L_{q,\omega}(\Omega)$, we have
$$
|f|^{q_0\mu}\in L_{\tilde q,\omega}(\Omega),\quad
|f|^{q_0}\in L_{\tilde q\mu,\omega}(\Omega),
$$
and by Lemma 3.1 of \cite{DK15S}, $f\in L_{q_0\mu}(\Omega)$. We first prove the a priori estimate \eqref{eq11.32}.
Let $(u,p) \in  W_{q,\omega}^1(\Omega)^d \times L_{q,\omega}(\Omega)$ be a solution to \eqref{eq11.28} satisfying $(p)_\Omega=0$. Then, $(u,p) \in  W_{q_0\mu}^1(\Omega)^d \times L_{q_0\mu}(\Omega)$.

By the properties (3) and (4) stated before Lemma \ref{lem6.5}, for each $Q\in \bC_n$ in the partitions,
there exist $r \in (0,\infty)$ and $x_0 \in \bar\Omega$ such that
\begin{equation}
                                \label{eq11.23}
Q \subset \Omega_r(x_0)\quad\text{and}
\quad |\Omega_r(x_0)| \le N |Q|,
\end{equation}
where $N$ depends on $d$, $R_0$, and $K$.
To apply Lemma \ref{lem6.5} (ii), we take 
$$
\ff=N^{-1}(|Du|+|p|),\quad \fh=N(|Du|+|p|),
$$ 
where $N=N(d, \delta, \mu, q_0) = N(d,\delta, R_0, K, q, [\omega]_{A_q})$, and
\begin{align*}
\fg(y) &= (\kappa^{-\frac 1 2} + \kappa \rho) \left[ \cM(|Du|^{q_0}+|p|^{q_0})(y)\right]^{\frac{1}{q_0}}
+ \kappa^{\frac{d}{q_0}} \left[ \cM(F^{q_0})(y)\right]^{\frac{1}{q_0}}\\
&\quad + \kappa^{\frac{d}{q_0}} \rho^{\frac{1}{q_0\nu}} \left[ \cM (|Du|^{q_0\mu}
+|p|^{q_0\mu})(y) \right]^{\frac{1}{q_0\mu}}+R_1^{-d}\kappa^{d}\left\||Du|+|p|\right\|_{L_1(\Omega)},
\end{align*}
where $F=|f_\alpha|+|g|$.
For $\ff^Q$, we consider two cases. When $\kappa r\le R_1$, we choose $\ff^Q=\cU$, where $\cU$ is from Lemma \ref{lem6.4}.
Thanks to \eqref{eq0919_01},  \eqref{eq3.33} with $q$ replaced by $q_0$, and \eqref{eq11.23}, we have that \eqref{eq10.53} holds. Otherwise, i.e., if $r>R_1/\kappa$, we take $\ff^Q=|Du|+|p|$. Then, by \eqref{eq11.23}, we have 
\begin{align*}
\dashint_Q |\ff^Q(x) - \left(\ff^Q\right)_Q| \,dx
&\le N\dashint_{\Omega_r(x_0)} |Du|+|p| \,dx\\
&\le N|\Omega_{R_1/\kappa}(x_0)|^{-1}\left\||Du|+|p|\right\|_{L_1(\Omega)}.
\end{align*}
Since $|\Omega_{R_1/\kappa}(x_0)|^{-1}\le NR_1^{-d}\kappa^{d}$, we still get that \eqref{eq10.53} holds. Therefore, the conditions in Lemma \ref{lem6.5} are satisfied, which yields
\begin{align*}
&\|Du\|_{L_{q,\omega}}+\|p\|_{L_{q,\omega}}
\le N\|\fg\|_{L_{q,\omega}}+N_1\left\||Du|+|p|\right\|_{L_1}\\
&\le N_1 R_1^{-d}\kappa^{d}\left\||Du|+|p|\right\|_{L_1}+N(\kappa^{-\frac 1 2} + \kappa \rho)
\|\cM(|Du|^{q_0}+|p|^{q_0})^{\frac{1}{q_0}}\|_{L_{q,\omega}}\\
&\quad +N\kappa^{\frac{d}{q_0}}\|\cM(F^{q_0})^{\frac{1}{q_0}}\|_{L_{q,\omega}}
+N\kappa^{\frac{d}{q_0}} \rho^{\frac{1}{q_0\nu}}
\|\cM(|Du|^{q_0\mu}+|p|^{q_0\mu})^{\frac{1}{q_0\mu}}\|_{L_{q,\omega}}\\
&= N_1 R_1^{-d}\kappa^{d}\||Du|+|p|\|_{L_1}+N(\kappa^{-\frac 1 2} + \kappa \rho)
\|\cM(|Du|^{q_0}+|p|^{q_0})\|_{L_{q/q_0,\omega}}^{\frac{1}{q_0}}\\
&\quad +N\kappa^{\frac{d}{q_0}}\|\cM(F^{q_0})\|^{\frac{1}{q_0}}_{L_{q/q_0,\omega}}
+N\kappa^{\frac{d}{q_0}} \rho^{\frac{1}{q_0\nu}}
\|\cM(|Du|^{q_0\mu}+|p|^{q_0\mu})\|^{\frac{1}{q_0\mu}}_{L_{q/{q_0\mu},\omega}},
\end{align*}
where $N = N(d, \delta, R_0, K, q, [\omega]_{A_q})$ and $N_1 = N(d, \delta,R_0, K, q, [\omega]_{A_q}, \omega(\Omega))$.
By the weighted maximal function theorem (see, for instance, \cite{MR0740173}), the right-hand side above is bounded by
\begin{align*}
&N_1 R_1^{-d}\kappa^{d}\||Du|+|p|\|_{L_1}+N(\kappa^{-\frac 1 2} + \kappa \rho)
\||Du|^{q_0}+|p|^{q_0}\|_{L_{q/q_0,\omega}}^{\frac{1}{q_0}}\\
&\quad +N\kappa^{\frac{d}{q_0}}\|F^{q_0}\|^{\frac{1}{q_0}}_{L_{q/q_0}}
+N\kappa^{\frac{d}{q_0}} \rho^{\frac{1}{q_0\nu}}
\||Du|^{q_0\mu}+|p|^{q_0\mu}\|^{\frac{1}{q_0\mu}}_{L_{q/{q_0\mu}}}\\
&=N_1 R_1^{-d}\kappa^{d}\||Du|+|p|\|_{L_1}+N(\kappa^{-\frac 1 2} + \kappa \rho)
\||Du|+|p|\|_{L_{q,\omega}}\\
&\quad +N\kappa^{\frac{d}{q_0}}\|F\|_{L_{q,\omega}}
+N\kappa^{\frac{d}{q_0}} \rho^{\frac{1}{q_0\nu}}
\||Du|+|p|\|_{L_{q,\omega}},
\end{align*}
where $N=N(d,\delta, R_0,K,q,[\omega]_{A_q})>0$ and $N_1 = N_1(d,\delta,R_0,K, q, [\omega]_{A_q}, \omega(\Omega)) > 0$.
Upon taking a sufficiently large $\kappa$, then a sufficiently small $\rho$, depending on $d$, $\delta$, $R_0$, $K$, $q$, and $[\omega]_{A_q}$, so that
$$
N(\kappa^{-\frac 1 2} + \kappa \rho)+N\kappa^{\frac{d}{q_0}} \rho^{\frac{1}{q_0\nu}}\le 1/2,
$$
we get 
\begin{equation}
							\label{eq0328_04}
\|Du\|_{L_{q,\omega}(\Omega)}\le N\|Du\|_{L_1(\Omega)}+N\|f_\alpha\|_{L_{q,\omega}(\Omega)}+N\|g\|_{L_{q,\omega}(\Omega)},
\end{equation}
where $N=N(d,\delta,R_0, R_1,K, q, [\omega]_{A_q}, \omega(\Omega))>0$. To bound the first term on the right-hand side, we notice that since $\Omega$ is bounded,  $f_\alpha,g\in L_{q_0}(\Omega)$ and
$$
\|f_\alpha\|_{L_{q_0}(\Omega)}\le N\|f_\alpha\|_{L_{q,\omega}(\Omega)},\quad
\|g\|_{L_{q_0}(\Omega)}\le N\|g\|_{L_{q,\omega}(\Omega)}.
$$
Thus, by Theorem \ref{thm1},
$$
\|Du\|_{L_{q_0}(\Omega)} + \|p\|_{L_{q_0}(\Omega)} \le N \|f_\alpha\|_{L_{q,\omega}(\Omega)} +N \|g\|_{L_{q,\omega}(\Omega)},
$$
where $N=N(d,\delta, R_0,R_1, K, q_0)$, provided that $\rho$ is sufficiently small depending only on $d$, $\delta$, $R_0$, $K$, and $q_0$. Combining this with \eqref{eq0328_04} yields the desired estimate \eqref{eq11.32}.

Next, we prove the solvability. By reverse H\"{o}lder's inequality for $A_p$ weights in spaces of homogeneous type proved in \cite[Theorem 3.2]{MS1981}, we have $\omega\in L_{1+\varepsilon}(\Omega)$, where $\varepsilon>0$ depends only on $d$, $R_0$, $K$, $q$, and $[\omega]_{A_q}$. Set $q_1=q/\varepsilon+q$. Then, by H\"older's inequality, it is easily seen that for any $f\in L_{q_1}(\Omega)$, we have $f\in L_{q,\omega}(\Omega)$ and
\begin{equation*}
\|f\|_{L_{q,\omega}(\Omega)}\le N\|f\|_{L_{q_1}(\Omega)},
\end{equation*}
where $N=N(d,R_0,K,q,[\omega]_{A_q}, \omega(\Omega))>0$.
For $k>0$, define $f^k_\alpha$ and $g^k$ as in \eqref{eq1.46}. Since $\Omega$ is bounded, we have
$$
f^k_\alpha, \, g^k\in L_{\infty}(\Omega)\subset L_{q_1}(\Omega).
$$
Then, by Theorem \ref{thm1}, there is a unique solution $(u_k,p_k)\in W^1_{q_1}(\Omega)^d\times L_{q_1}(\Omega)$ satisfying $(p_k)_\Omega=0$ and \eqref{eq1.48}. Thus, by the reasoning above, we have $(u_k,p_k) \in  W_{q,\omega}^1(\Omega)^d \times L_{q,\omega}(\Omega)$.
As $f_\alpha,g\in L_{q_0}(\Omega)$, we have $|f_\alpha|,|g|<\infty$ a.e. in $\Omega$. By the dominated convergence theorem, it is easily seen that
$$
f^k_\alpha\to f_\alpha,\quad g^k\to g\quad \text{in}\quad L_{q,\omega}(\Omega)
$$
as $k\to \infty$. Moreover, $(g^k)_\Omega \to 0$ as $k\to\infty$. Therefore, by \eqref{eq11.32}, $\{(u_k,p_k)\}$ is a Cauchy sequence in $\mathring W_{q,\omega}^1(\Omega)^d \times L_{q,\omega}(\Omega)$. Let $(u,p)\in \mathring W_{q,\omega}^1(\Omega)^d \times L_{q,\omega}(\Omega)$ be the limit of the sequence. Following the last part of the proof of Theorem \ref{thm1}, we see that $(u,p)$ satisfies \eqref{eq11.28} and $(p)_\Omega=0$.
Finally, the uniqueness follows from the a priori estimate \eqref{eq11.32}. The theorem is proved.
\end{proof}

\bibliographystyle{plain}

\begin{thebibliography}{10}

\bibitem{AGZ11}
H. Abidi, G. Gui, and P. Zhang.
\newblock On the decay and stability of global solutions to the 3D inhomogeneous Navier-Stokes equations.
\newblock {\em Comm. Pure Appl. Math.},64(6):832--881, 2011.

\bibitem{MR2263708}
Gabriel Acosta, Ricardo~G. Dur{\'a}n, and Mar{\'{\i}}a~A. Muschietti.
\newblock Solutions of the divergence operator on {J}ohn domains.
\newblock {\em Adv. Math.}, 206(2):373--401, 2006.

\bibitem{MR0740173}
Hugo Aimar and Roberto~A. Mac{\'{\i}}as.
\newblock Weighted norm inequalities for the {H}ardy-{L}ittlewood maximal
  operator on spaces of homogeneous type.
\newblock {\em Proc. Amer. Math. Soc.}, 91(2):213--216, 1984.

\bibitem{AHMNT2014}
Jonas Azzam, Steve Hofmann, Jos\'{e}~Mar\'{i}a Martell, Kaj Nystr\"{o}m, and
  Tatiana Toro.
\newblock A new characterization of chord-arc domains.
\newblock {\em arXiv:1406.2743}.

\bibitem{BMR07}
M. Bul{\'{\i}}{\v{c}}ek, J. M{\'a}lek, and K. R. Rajagopal.
\newblock Navier's slip and evolutionary Navier-Stokes-like systems with pressure and shear-rate dependent viscosity.
\newblock {\em Indiana Univ. Math. J.}, 56(1):51--85, 2007.

\bibitem{BS16}
Sun-Sig Byun and Hyoungsuk So.
\newblock Weighted estimates for generalized steady {S}tokes systems in nonsmooth domains.
\newblock {\em arXiv:1609.03290}.


\bibitem{CL15}
Jongkeun Choi and Ki-Ahm Lee.
\newblock The green function for the stokes system with measurable
  coefficients.
\newblock {\em arXiv:1503.07290}.

\bibitem{MR1096400}
Michael Christ.
\newblock A {$T(b)$} theorem with remarks on analytic capacity and the {C}auchy
  integral.
\newblock {\em Colloq. Math.}, 60/61(2):601--628, 1990.

\bibitem{DM04}
M. Dindo{\v{s}} and M. Mitrea.
\newblock The stationary Navier-Stokes system in nonsmooth manifolds: the Poisson problem in Lipschitz and $C^1$ domains.
\newblock {\em Arch. Ration. Mech. Anal.} 174(1):1--47, 2004.

\bibitem{DK15S}
Hongjie Dong and Doyoon Kim.
\newblock {$L_q$}-estimates for stationary stokes system with coefficients
  measurable in one direction.
\newblock {\em arXiv:1604.02690}.

\bibitem{DK15}
Hongjie Dong and Doyoon Kim.
\newblock On {$L_p$}-estimates for elliptic and parabolic equations with
  {$A_p$} weights.
\newblock {\em Trans. Amer. Math. Soc. (to appear), arXiv:1603.07844}.

\bibitem{MR2835999}
Hongjie Dong and Doyoon Kim.
\newblock Higher order elliptic and parabolic systems with variably partially
  {BMO} coefficients in regular and irregular domains.
\newblock {\em J. Funct. Anal.}, 261(11):3279--3327, 2011.

\bibitem{MR2771670}
Hongjie Dong and Doyoon Kim.
\newblock On the {$L_p$}-solvability of higher order parabolic and elliptic
  systems with {BMO} coefficients.
\newblock {\em Arch. Ration. Mech. Anal.}, 199(3):889--941, 2011.

\bibitem{MR2800569}
Hongjie Dong and Doyoon Kim.
\newblock Parabolic and elliptic systems in divergence form with variably
  partially {BMO} coefficients.
\newblock {\em SIAM J. Math. Anal.}, 43(3):1075--1098, 2011.

\bibitem{MR3266252}
Hongjie Dong and Doyoon Kim.
\newblock On the impossibility of {$W_p^2$} estimates for elliptic equations
  with piecewise constant coefficients.
\newblock {\em J. Funct. Anal.}, 267(10):3963--3974, 2014.

\bibitem{MR975121}
E.~B. Fabes, C.~E. Kenig, and G.~C. Verchota.
\newblock The {D}irichlet problem for the {S}tokes system on {L}ipschitz
  domains.
\newblock {\em Duke Math. J.}, 57(3):769--793, 1988.

\bibitem{FMR05} M. Franta, J. M{\'a}lek, K. R. Rajagopal.
\newblock On steady flows of fluids with pressure- and shear-dependent viscosities.
\newblock {\em Proc. R. Soc. Lond. Ser. A Math. Phys. Eng. Sci.}, 461(2055):651--670, 2005.

\bibitem{MR1313554}
G.~P. Galdi, C.~G. Simader, and H.~Sohr.
\newblock On the {S}tokes problem in {L}ipschitz domains.
\newblock {\em Ann. Mat. Pura Appl. (4)}, 167:147--163, 1994.

\bibitem{MR641818}
M.~Giaquinta and G.~Modica.
\newblock Nonlinear systems of the type of the stationary {N}avier-{S}tokes
  system.
\newblock {\em J. Reine Angew. Math.}, 330:173--214, 1982.

\bibitem{MR717034}
Mariano Giaquinta.
\newblock {\em Multiple integrals in the calculus of variations and nonlinear
  elliptic systems}, volume 105 of {\em Annals of Mathematics Studies}.
\newblock Princeton University Press, Princeton, NJ, 1983.

\bibitem{MR1331981}
David Jerison and Carlos~E. Kenig.
\newblock The inhomogeneous {D}irichlet problem in {L}ipschitz domains.
\newblock {\em J. Funct. Anal.}, 130(1):161--219, 1995.

\bibitem{MR1446617}
Carlos~E. Kenig and Tatiana Toro.
\newblock Harmonic measure on locally flat domains.
\newblock {\em Duke Math. J.}, 87(3):509--551, 1997.

\bibitem{MR2540989}
N.~V. Krylov.
\newblock Second-order elliptic equations with variably partially {VMO}
  coefficients.
\newblock {\em J. Funct. Anal.}, 257(6):1695--1712, 2009.

\bibitem{MR563790}
N.~V. Krylov and M.~V. Safonov.
\newblock A property of the solutions of parabolic equations with measurable
  coefficients.
\newblock {\em Izv. Akad. Nauk SSSR Ser. Mat.}, 44(1):161--175, 239, 1980.

\bibitem{LS75}
O. A. Lady{\v{z}}enskaya  and V. A. Solonnikov.
\newblock The unique solvability of an initial-boundary value problem for viscous incompressible inhomogeneous fluids.
\newblock Boundary value problems of mathematical physics, and related questions of the theory of functions, 8.
\newblock {\em Zap. Nau\v cn. Sem. Leningrad. Otdel. Mat. Inst. Steklov. (LOMI)}, 52:52--109, 218--219,  (1975).

\bibitem{Li96}
Pierre-Louis Lions,
\newblock {\em Mathematical topics in fluid mechanics. Vol. 1.} Incompressible models.
\newblock  Oxford Lecture Series in Mathematics and its Applications, 3. Oxford Science Publications. The Clarendon Press, Oxford University Press, New York, 1996. xiv+237 pp.

\bibitem{MS1981}
Roberto~A. Mac{\'{\i}}as and Carlos~A. Segovia.
\newblock {\em A well behaved quasi-distance for spaces of homogeneous type},
  volume~32 of {\em Trabajos de Matem{\'a}tica}.
\newblock Inst. Argentino Mat., 1981.

\bibitem{MR0159110}
Norman~G. Meyers.
\newblock An {$L\sp{p}$}-estimate for the gradient of solutions of second order
  elliptic divergence equations.
\newblock {\em Ann. Scuola Norm. Sup. Pisa (3)}, 17:189--206, 1963.

\bibitem{MR2135732}
Roberto Monti and Daniele Morbidelli.
\newblock Regular domains in homogeneous groups.
\newblock {\em Trans. Amer. Math. Soc.}, 357(8):2975--3011, 2005.

\bibitem{MR579490}
M.~V. Safonov.
\newblock Harnack's inequality for elliptic equations and {H}\"older property
  of their solutions.
\newblock {\em Zap. Nauchn. Sem. Leningrad. Otdel. Mat. Inst. Steklov. (LOMI)},
  96:272--287, 312, 1980.
\newblock Boundary value problems of mathematical physics and related questions
  in the theory of functions, 12.

\end{thebibliography}

\def\cprime{$'$}


\end{document}